\documentclass[a4paper,12pt,reqno]{amsart}

\usepackage{graphicx}
\usepackage{float}
\usepackage{amsmath}
\usepackage{amssymb}
\usepackage{comment}
\usepackage{harvard}

\numberwithin{equation}{section}

\theoremstyle{plain}                    % default
\newtheorem{thm}{Theorem}[section]
\newtheorem{lem}{Lemma}[section]
\newtheorem{cor}{Corollary}[section]
\newtheorem{prop}{Proposition}[section]

\theoremstyle{definition}
\newtheorem{defn}{Definition}[section]

\newtheorem{rem}{Remark}[section]               %\renewcommand{\therem}{}

\theoremstyle{remark}

\def\R{{\mathbb R}}

\def\N{{\mathbb N}}
\def\1{\mbox{I\hspace{-.6em}1}}

\newcommand{\be}{\begin{equation}}
\newcommand{\bd}{\begin{displaymath}}
\newcommand{\ed}{\end{displaymath}}
\newcommand{\bea}{\begin{eqnarray}}
\newcommand{\eea}{\end{eqnarray}}
\newcommand{\bean}{\begin{eqnarray*}}
\newcommand{\eean}{\end{eqnarray*}}

\providecommand{\abs}[1]{\lvert #1 \rvert}
\providecommand{\norm}[1]{\lVert #1 \rVert}
\providecommand{\eps}{\varepsilon}
\renewcommand{\phi}{\varphi}
\renewcommand{\theta}{\vartheta}

\providecommand{\floor}[1]{\lfloor #1 \rfloor}
\allowdisplaybreaks[1]

\begin{document}

\thispagestyle{empty}

\begin{center}
{\large \sc Asymptotic Equivalence for Nonparametric Regression with Non-Regular Errors}
 \end{center}

\vspace*{1cm}

\begin{tabular}{ccc}
Alexander Meister &\hspace{0.4cm} & Markus Rei{\ss} \\ \begin{small} Institut f\"ur Mathematik \end{small} && \begin{small} Institut f\"ur Mathematik \end{small} \\ \begin{small} Universit\"at Rostock \end{small} && \begin{small} Humboldt-Universit\"at zu Berlin \end{small} \\  \begin{small} Ulmenstra{\ss}e 69 \end{small} && \begin{small} Unter den Linden 6 \end{small} \\
\begin{small} 18051 Rostock, Germany \end{small} && \begin{small} 10099 Berlin, Germany \end{small} \\[1mm] \begin{small}
e-mail: alexander.meister@uni-rostock.de \end{small} && \begin{small} e-mail: mreiss@math.hu-berlin.de \end{small}
\end{tabular}

 \vspace*{2.5cm}

 \begin{center}
 {\bf Abstract}
 \end{center}
\begin{small}
Asymptotic equivalence in Le Cam's sense for nonparametric regression experiments is extended to the case of non-regular error densities, which have jump discontinuities at their endpoints. We prove asymptotic equivalence of such regression models and the observation of two independent Poisson point processes which contain the target curve as the support boundary of its intensity function. The intensity of the point processes is of order of the sample size $n$ and involves the jump sizes as well as the design density. The statistical model significantly differs from regression problems with Gaussian or regular errors, which are known to be asymptotically equivalent to Gaussian white noise models.
 \end{small}

\vspace*{1.5cm}

%\noindent \today

\vfill

\footnoterule \vspace{3mm}


%\noindent {\sl Running title:} Asymptotic Equivalence for Regression with Non-Regular Errors.

\newpage
\pagestyle{headings} \setcounter{page}{1}

\section{Introduction} \label{1}

The goal of transforming nonparametric regression models into asymptotically equivalent statistical experiments, which describe continuous observations of a stochastic process, has stimulated considerable research activity in mathematical statistics. The continuous design in these limiting models simplifies the asymptotic analysis and makes statistical procedures more transparent because in the regression case the discrete design points generate distracting approximation errors.
Most papers so far establish asymptotic equivalence of certain nonparametric regression models with nonparametric Gaussian shift experiments.
%on this topic are concerned with equivalence results of curve estimation (e.g. density or regression %estimation) and a white noise model.
In that Gaussian white noise experiment, a process is observed which contains the target function in its drift and a blurring Wiener process which is scaled with a factor of order $n^{-1/2}$, where $n$ denotes the original sample size. The basic equivalence result for standard Gaussian regression with deterministic design has been established by \citeasnoun{BL96}. Afterwards, many important extensions have been achieved. The case of random design for univariate design has been treated by \citeasnoun{BCLZ02}. \citeasnoun{C07} considers the case of unknown error variance and design density; and \citeasnoun{R08} extends the results to the multivariate setting. Recently, the model with dependent regression errors has been investigated in \citeasnoun{C10}. The work by \citeasnoun {GN98} is the first to consider the important case of non-Gaussian errors which are, however, supposed to be included in an exponential family. Such classes of error distributions are also studied in \citeasnoun{BCZ10} where the regression error is supposed to be non-additive. General regular distributions for the additive error variables are covered in \citeasnoun{GN02} where only slightly more than standard Hellinger differentiablity is required for the error density.

On the other hand, when allowing for jump discontinuities of the error density, the situation changes completely. Standard examples include uniform or exponential error densities. These types of error distributions are non-regular and we know from parametric theory that better rates of convergence and non-Gaussian limit distributions can be expected. The faster convergence rates are attained only by specific estimators, e.g. employing extreme value statistics in their construction instead of local averaging statistics. The Nadaraja-Watson estimator and the local polynomial estimators are procedures of that latter type, which can be improved significantly under non-regular errors.  \citeasnoun{MW10} establish  improved minimax rates for regression functions which satisfy some H\"older condition. \citeasnoun{HK09} derive a rigorous theory for the optimal convergence rates for nonparametric regression under non-regular errors and smoothness constraints up to regularity one on the target regression function. Their nonparametric minimax rates in dimension one are of the form $n^{-s/(s+1)}$ for H\"older regularity $s$, which is faster than the usual $n^{-s/(2s+1)}$-rate for regular regression, but slower than $n^{-2s/(2s+1)}$, the squared regular rate in analogy with the parametric rates. At first sight, this is counter-intuitive, but may be explained by a Poisson instead of Gaussian limiting law. Many applications of non-regular regression models occur in the field of econometrics, see \citeasnoun{CH04} for an overview and a precise asymptotic investigation of the parametric likelihood ratio process. Irregular regression problems are also closely related to nonparametric boundary estimation in image reconstruction, see the monograph of \citeasnoun{KT93}. Considerable interest has also found the problem of frontier estimation, see \citeasnoun{GMPS99} and the references therein.

%However, to our best knowledge, a theory for asymptotic equivalence of nonparametric regression with non-regular errors and a stochastic process model has been missing in literature so far.

In \citeasnoun{JM94} weak asymptotic equivalence of the extreme order statistics of a one-dimensional localization problem with non-regular errors and a Poisson point process model is derived in a parametric setup. Also for the precise asymptotic analysis of regression experiments with non-regular errors the use of Poisson point processes and random measures turn out to be useful, see e.g. \citeasnoun{K01} for parametric linear models and \citeasnoun{CH04} for general parametric regression, yet a precise and nonparametric statement lacks. We intend to fill this gap by rigorously proving asymptotic equivalence of nonparametric regression experiments with non-regular errors with a Poisson point process (PPP) model. Therein the target parameter occurs as the boundary curve of the intensity function. Hence, the Gaussian structure of the process experiment is not kept; nor is the scaling factor $n^{-1/2}$ which will be changed into $n^{-1}$ in agreement with the parametric rate. For a comprehensive review on  PPP and their statistical inference we refer to \citeasnoun{K91} and \citeasnoun{K98}. They discuss  image reconstruction from laser radar as a practical application of support estimation of the intensity function of a PPP, which corresponds to identifying the target parameter in our PPP experiment. The asymptotic equivalence result therefore links interesting inference questions in both models which might prove useful in both directions.
%In nonparametric experiments, equivalence aspects with respect to Poisson process models are not as well understood as with respect to Ito processes. In fact the only paper of this first category which we are aware of is \citeasnoun{BCLZ04} where the authors prove asymptotic equivalence of density estimation and a one-dimensional Poisson process having the target density multiplied by the sample size of the corresponding density estimation experiment as its intensity function.

For the basic concept of asymptotic equivalence of statistical experiments we refer to \citeasnoun{LC64} and  \citeasnoun{LCY00}. To grasp the impact let us just mention that asymptotic equivalence between two sequences of statistical models transfers asymptotical risk bounds for any inference problem from one model to the other, at least for bounded loss functions. Moreover, asymptotic equivalence remains valid for the sub-experiments obtained by restricting the parameter class so that we shall also cover smoother nonparametric or just parametric regression problems.

The paper is organized as follows. In Section \ref{2} we introduce our models, state our main result in Theorem \ref{T:1} and give a constructive description of the equivalence maps. In Section \ref{2.0} we construct pilot estimators of the target functions which will be employed to localize the model in Section \ref{3} and \ref{5}. The findings of Section \ref{4} yield asymptotic equivalence of the PPP experiment and the regression model when the target functions are changed into approximating step functions. In Section \ref{6} all the results are combined to complete the proof of Theorem \ref{T:1}. Section \ref{7} discusses limitations and extensions of the results and gives a geometric explanation of the unexpected nonparametric minimax rate for H\"older classes.

\section{Model and main result} \label{2}

In this section we specify the statistical experiments under consideration.
First we define the joint parameter space $\Theta$ of both the regression and
the PPP experiment, imposing standard smoothness constraints on the target
function.

\begin{defn}\label{defTheta}
For some constants $C_\Theta>0$ and $\alpha\in(0,1]$ the parameter set $\Theta$
consists of all functions $\theta:[0,1]\to\R$ which are twice continuously
differentiable on $[0,1]$ with $\|\theta\|_\infty\leq C_\Theta$ and
$\|\theta''\|_\infty \leq C_\Theta$ and where the second derivative satisfies
the H\"older condition
$$ \big|\theta''(x) - \theta''(y)\big|\, \leq \, C_\Theta |x-y|^\alpha\,, \quad \forall x,y\in [0,1]\,. $$
\end{defn}

In the regression model $\Theta$ represents the collection of all admitted
regression functions. This parameter space will remain unchanged for all
experiments considered here.

\begin{defn}\label{defA}
We define the statistical experiment ${\mathcal A}_n$ in which the data
$Y_{j,n}$, $j=1,\ldots,n$, with
\begin{equation} \label{eq:A}
Y_{j,n} \, = \, \theta(x_{j,n}) + \varepsilon_{j,n}
\end{equation}
are observed. The deterministic design points $x_{1,n},\ldots,x_{n,n}\in[0,1]$
are assumed to satisfy
\begin{equation} \label{eq:con_X_0}
x_{j,n} \, = \, F_D^{-1}\big((j-1)/(n-1)\big)\,,
\end{equation}
where the distribution function $F_D:[0,1]\to[0,1]$ possesses a Lipschitz
continuous Lebesgue density $f_D$ which is uniformly bounded away from zero.
The regression errors $\varepsilon_{j,n}$ are assumed to be i.i.d. with  error
density $f_\eps:[0,1]\to\R^+$, which is Lipschitz continuous and strictly
positive.
%, satisfying for some finite constant $C_\eps>0$
%\[ \sup_{t\not=s}\frac{|f_\eps(t)-f_\eps(s)|}{|t-s|} +
%\sup_t|f_\eps(t)|\, \leq \, C_\eps.
%\]
\end{defn}

The conditions on the design are adopted from \citeasnoun{BL96}. They imply
that
\begin{equation} \label{eq:con_X} d^{-1}/n \, \leq \, x_{j+1,n} - x_{j,n} \, \leq \, d/n\,, \end{equation}
\noindent for all $n\in \N$, $j=1,\ldots,n$ and a finite positive constant $d$.

The error model describes the class of densities which are supported on
$[-1,1]$, regular within $(-1,1)$ and which have jumps at their left and
right endpoints. Note that by constant extrapolation the density $f_\eps$ on
$[-1,1]$ can always be written as
$$ f_\varepsilon(x) \, = \, 1_{[-1,1]}(x)\cdot \varphi(x)\,, $$
\noindent with a strictly positive Lipschitz continuous function
$\varphi:\R\to\R$ satisfying for some constant $C_\eps>0$
\begin{align} \label{eq:error0}
 \sup_{t\not=s}\frac{|\varphi(t)-\varphi(s)|}{|t-s|} + \sup_t|\varphi(t)|\, \leq \, C_\eps.
\end{align}
Instead of constant extrapolation, $\varphi$ may alternatively be continued such that $\varphi\in L_1(\R)$ holds in addition.  

Hence, experiment ${\mathcal A}_n$ describes a non-regular nonparametric
regression model. We believe that the regularity condition on $f_\eps$ in the
interior $(-1,1)$ can be substantially relaxed, but at the cost of more
involved estimation techniques. We have restricted our
consideration to the specific interval $[-1,1]$ for convenience.

In the PPP model the target function $\theta$ occurs as upper and lower
boundary curves of the intensity functions of two independent Poisson point
processes $X_1$ and $X_2$.

\begin{defn}\label{defB}
For functions $\theta\in\Theta$, the design density $f_D$ and the noise density $f_\eps$ from above we define the experiment ${\mathcal B}_n$ in which we observe two independent
Poisson point processes $X_j$, $j=1,2$, on the rectangle $S = [0,1]\times
[-C_\Theta-1,C_\Theta+1]\subset\R^2$ with respective intensity functions
\begin{align}  \nonumber
\lambda_1(x,y) & \, = \, f_D(x)\cdot 1_{[-C_\Theta-1,\theta(x)]}(y)\cdot n f_\varepsilon(1), \\
\label{eq:B} \lambda_2(x,y) & \, = \, f_D(x)\cdot
1_{[\theta(x),C_\Theta+1]}(y)\cdot n f_\varepsilon(-1)\,,
\end{align}
\noindent for all $(x,y)\in S$.
\end{defn}

Each realisation $X_j$ represents a measure mapping from the Borel subsets of
$S$ to $\N \cup\{0\}$. Equivalently, $X_j(\cdot)/X_j(S)$ may be characterized
by a two-dimensional discrete probability distribution, see \citeasnoun{K91} or
\citeasnoun{K98} for more details on PPP. Thus, the underlying action space can
be taken as a Polish space (e.g. the separable Banach space $L_1(S)$) such that
asymptotic equivalence can be established by Markov kernels.

Figure \ref{Fig1} shows on the left the regression function $\theta(x)=\frac{3}{10}x\cos(10x)$ and corresponding $n=100$ equidistant observations of ${\mathcal A}_n$ corrupted by uniform noise on $[-1,1]$. A realisation of the equivalent PPP model ${\mathcal B}_n$ is shown on the right, with '+', '-' indicating point masses of $X_2$ and $X_1$, respectively.

\begin{figure} 
\centering
\includegraphics[width=7cm]{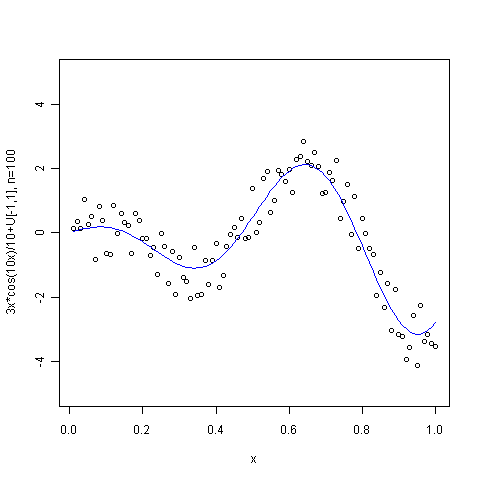}
\includegraphics[width=7cm]{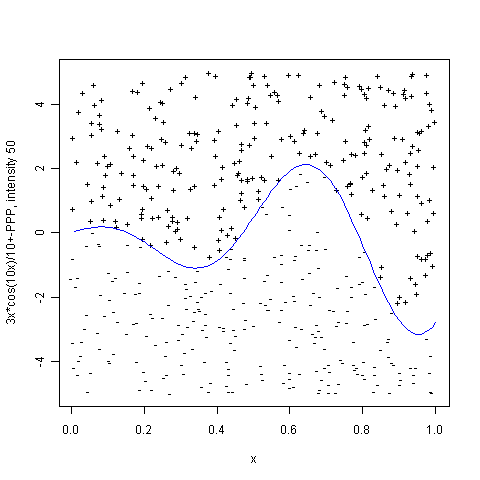}
\caption{Left: Regression model ${\mathcal A}_n$ with uniform $U[-1,1]$ errors.  Right: Equivalent Poisson point process model ${\mathcal B}_n$}
\label{Fig1}
\end{figure}

We may conceive $X_j$ as the random point measure
$\sum_{k=1}^{N_j}\delta_{(x_k^j,y_k^j)}$ where $N_j$ is drawn from a
Poisson-distribution with intensity $\norm{\lambda_j}_{L^1(S)}$ and the $(x_k^j,y_k^j)$
are drawn according to the bivariate density $\lambda_j/\norm{\lambda_j}_{L^1(S)}$. The vertical bounds $\pm(C_\Theta+1)$ for the
domain $S$ are non-informative for $\theta\in\Theta$, but the boundedness
avoids technicalities. The equivalent unbounded PPP can be described by infinite random point measures $\sum_{k=1}^{\infty}\delta_{(x_k^j,y_k^j)}$
where the $x_k^j$ are drawn according to the density $f_D$ and \[\textstyle y_k^1=\theta(x_k^1)-(nf_\eps(1))^{-1}\sum_{l=1}^kz_l^1, \quad y_k^2=\theta(x_k^2)+(nf_\eps(-1))^{-1}\sum_{l=1}^kz_l^2
\]
 holds with exponentially distributed $(z_k^j)$ of mean one (all independent). In this form, the PPP already appears in
\citeasnoun{K01}, yielding the limiting law for parametric estimators in the nonregular linear
model.

%can be described by the joint observation of the bounded PPP as above as well as independent PPP on the rectangles $[-1,1]\times [(2j-1)(C_\Theta+1),(2j+1)(C_\Theta+1)]$, integer $j\neq 0$, with the intensity functions $f_D(x) n f_\varepsilon(\pm 1)$ on these domains.

\begin{comment}
The corresponding unbounded PPP can be described by infinite random point measures $\sum_{k=1}^{\infty}\delta_{(x_k^j,y_k^j)}$
where the $x_k^j$ are drawn according to the density $f_D$ and the $y_k^1$
follow a shifted and mirrored exponential distribution on
$(-\infty,\theta(x_k^1)]$ (resp. the $y_k^2$ shifted exponential on
$[\theta(x_k^2),\infty)$). In this form, the PPP already appears in
\citeasnoun{K01} as a limiting parametric model for the nonregular linear
model.
\end{comment}

We present the main result of this work in the following theorem.

\begin{thm} \label{T:1}
The statistical experiments ${\mathcal A}_n$ and ${\mathcal B}_n$ are
asymptotically equivalent in Le Cam's sense as $n\to \infty$.
\end{thm}

This asymptotic equivalence is achieved constructively by consecutive
invertible (in law) and parameter-independent mappings of the data, which
generate new experiments where the observation laws are shown to be
asymptotically close (uniformly over $\theta$ in total variation norm). In
order to highlight the main ideas in the subsequent proof and to indicate how
to use our theoretical result in practice, let us give an algorithmic
description of these equivalence mappings leading from experiment ${\mathcal
A}_n$ to experiment ${\mathcal B}_n$ (in the version with unbounded domain).
\begin{enumerate}
\item Take the data $Y_{j,n}$, $j=1,\ldots,n$, from experiment ${\mathcal A}_n$.
\item Split the data and bin one part: consider the odd indices
$J_n:=\{1,3,\ldots,\linebreak 2\lceil n/2\rceil -1\}$ and intervals
$I_k=[k/m,(k+1)/m)$ with some appropriate $m$. Put ${\bf X}_1=(Y_{j+1,n})_{j\in J_n\backslash
\{n\}}$ and $\bar {\bf Z}=(\bar Z_j)_{j\in J_n}$ with
    $$\bar Z_j=Y_{j,n}-\hat\theta_1(\xi_{j})-\hat\theta_1'(\xi_{j})(x_j-\xi_j),\quad j\in J_n,$$
    where $\xi_j$ is the centre of that interval $I_k$ with $x_{j,n}\in
    I_k$ and where $\hat\theta_1$ is a (good) estimator of $\theta$ based
    on the data ${\bf X}_1$.
\item Consider the local extremes in $\bar {\bf Z}$, i.e. $s_k=\min(\bar Z_k)$, $S_k=\max(\bar Z_k)$, $k=0,\ldots,m-1$.
\item Use $\hat\theta$ on the data ${\bf X}_1$ again to transform $s_k''=s_k+\hat\theta_1(\xi_k)+1$,
$S_k''=S_k+\hat\theta_1(\xi_k)-1$.

\item Randomization to build PPP $X_l$, $X_u$: on each interval $I_k$ generate $(x_k^l,y_k^l)$
with $x_k^l$ having the density $f_k = f_D 1_{I_k} / \int_{I_k} f_D$ independent of everything else and
$y_k^l=S_k''-\hat\theta_1'(x_k^l)(\xi_k-x_k^l)$; define the PPP $X_l$ where independently on each $I_k$
we observe a point measure in $(x_k^l,y_k^l)$ plus independently (conditionally on $S_k''$, $\hat\theta_1'$) a PPP
with intensity
$$\textstyle \frac n2 f_\varepsilon(1)(m\int_{I_k}f_D){\bf 1}\{x\in
I_k,\,y\le S_k''-\hat\theta_1'(x)(\xi_k-x)\};
$$
analogously generate $x_k^u$ with the density $f_k$ independently,
$y_k^u=s_k''-\hat\theta_1'(x_k^u)(\xi_k-x_k^u)$ and use the intensity
$$\textstyle \frac n2 f_\varepsilon(-1)(m\int_{I_k}f_D){\bf 1}\{x\in I_k,\,y\ge
s_k''-\hat\theta_1'(x)(\xi_k-x)\}
$$
to build $X_u$ independently conditionally on $s_k''$, $\hat\theta_1'$.

\item Use a (good) estimator $\hat\theta_2$ based on the PPP data ${\bf X}_2=(X_l,X_u)$
and redo steps (2)-(5) to transform ${\bf X}_1$ via $\bar
Z_{j+1}=Y_{j+1,n}-\hat\theta_2(\xi_{j+1})-\hat\theta_2'(\xi_{j+1})(x_{j+1}-\xi_{j+1})$,
$j\in J_n$, to another couple $(X_l',X_u')$ of PPP; the final PPP are
obtained by $X_1=X_l+X_l'$, $X_2=X_u+X_u'$.
\end{enumerate}

In this algorithmic description we could do without substracting and adding the
pilot estimator itself (i.e., only use the derivative) in steps (2) and (4),
but in the proof this localization permits an easy sufficiency argument for the
local extremes. Put in a nutshell, the asymptotic equivalence is achieved by
considering block-wise extreme values in the regression experiment, in
conjunction with a pre- and post-processing procedure (localization step)
performing a linear correction on each block. The easier block-wise constant
approximation approach by \citeasnoun{BL96} does not work here since we need a
much higher approximation order.

Throughout we shall write const. for a generic positive constant which may
change its value from line to line and does not depend on the parameter
$\theta$ nor on the sample size $n$. Similarly, the Landau symbols $O$, $o$ and
the asymptotic order symbol $\asymp$ will denote uniform bounds with respect to
$\theta$ and $n$.

\section{Pilot estimators} \label{2.0}

In order to prove Theorem \ref{T:1} a localization strategy is required as in
\citeasnoun{N96} for the density estimation problem. To that end we construct
pilot estimators of the target function $\theta$ and its derivative in both,
experiments ${\mathcal A}_n$ and ${\mathcal B}_n$.

Let us fix the estimation point $x_0\in [0,1]$ and apply a local polynomial estimation approach. We introduce the neighbourhood $U_h=[x_0-h,x_0+h]$ for $x_0\in[h,1-h]$ and the one-sided analogue $U_h=[0,2h]$ for $x_0\in[0,h)$, $U_h=[1-2h,1]$ for $x_0\in(1-h,1]$. We introduce the set $\Pi:=\Pi_2(U_h)$ of quadratic polynomials on $U_h$.
Standard approximation theory (by a Taylor series argument) gives for $h\downarrow 0$
\[ \gamma_{h}:=\sup_{\theta\in\Theta}\min_{p\in\Pi}\max_{x\in U_h}\big(h^{-(2+\alpha)}\abs{\theta(x)-p(x)} + h^{-(1+\alpha)}\abs{\theta'(x)-p'(x)}\big)\le \text{const.}<\infty\,,
\]
\noindent where the constant does not depend on $h$.

\begin{defn}
We call $\hat\theta\in \Pi$ in experiment ${\mathcal A}_n$ locally admissible
at $x_0$ if
$$\max_{j: x_{j,n}\in U_h}
|Y_{j,n}-\hat\theta(x_{j,n})|\le 1+\gamma_hh^{2+\alpha}$$ holds. Similarly, in
experiment ${\mathcal B}_n$  we call $\hat\theta\in \Pi$ locally admissible at
$x_0$ if
$$X_1(\{x\in U_h,\,y>\hat\theta(x)+\gamma_hh^{2+\alpha}\})=0\text{ and }X_2(\{x\in
U_h,\,y<\hat\theta(x)-\gamma_hh^{2+\alpha}\})=0
$$ hold. Our estimator
$\hat{\theta}_{n,h}(x_0)$ is just any locally admissible $\hat\theta_{n,h}\in
\Pi$, evaluated at $x_0$ and selected as a measurable function of the data (by
the measurable selection theorem).
\end{defn}

Note that the by $\gamma_h$ enlarged band size guarantees that
$\hat\theta_{n,h}$ exists since the minimizer $\theta_h\in\Pi$ in the
definition of $\gamma_h$ is eligible. The following result gives the pointwise
risk bounds for the regression function and its derivative with orders
$O(n^{-s/(s+1)})$ and $O(n^{-(s-1)/(s+1)})$, respectively, where $s=2+\alpha$
denotes the regularity in a H\"older class. As an application of our asymptotic
equivalence we shall show in Section \ref{SecLB} below the optimality of these rates in a minimax
sense. The upper bound proof relies on entropy arguments and norm equivalences
for polynomials and could be easily extended to more general local polynomial
estimation and $L^p$-loss functions.

\begin{prop} \label{Prop:0}
Select the bandwidth $h$ such that $h \asymp n^{-1/(3+\alpha)}$. Then we have
in experiment ${\mathcal A}_n$ as well as in experiment ${\mathcal B}_n$
$$ \sup_{\theta\in\Theta} \sup_{x_0\in [0,1]} E_\theta\big(n^{2(2+\alpha)/(3+\alpha)}\big|\hat{\theta}_{n,h}(x_0) - \theta(x_0)\big|^2+n^{2(1+\alpha)/(3+\alpha)}\big|\hat{\theta}_{n,h}'(x_0) - \theta'(x_0)\big|^2\big) \, \leq \, \text{const.}$$
\end{prop}

\noindent {\it Proof of Proposition \ref{Prop:0}:}
We shall need the following bounds in $\Pi=\Pi_2(U_h)$ from \citeasnoun{DL93}:  $\norm{p}_{L^\infty(U_h)}\le 8h^{-1}\norm{p}_{L^1(U_h)}$ (their Theorem IV.2.6); $\norm{p'}_{L^\infty(U_h)}\le c_0 h^{-1}\norm{p}_{L^\infty(U_h)}$ (their Thm. IV.2.7); their proof of Thm. IV.2.6 establishes $\abs{p(x)}\ge (1-4(x-x_M)/h)\norm{p}_\infty$ for $x_M:=\text{argmax}_{x\in U_h}\abs{p(x)}$  and $x_M\le x< x_M+h/4$, assuming without loss of generality that $x_M$ lies in the left half of $U_h$, such that uniformly over $x_0$
\[\norm{p}_{n,h,1}:=\frac{1}{nh}\sum_{x_{j,n}\in U_h}\abs{p(x_{j,n})}\geq\text{const.}\cdot \abs{p(x_M)}=\text{const.}\cdot \norm{p}_{L^\infty(U_h)}\]
is derived.

Let us start with considering the regression experiment ${\mathcal A}_n$. We
apply a standard chaining argument in the finite-dimensional space $\Pi$
together with an approximation argument.  From above we have
$\norm{p}_{L^\infty(U_h)}/\norm{p}_{n,h,1}\asymp 1$ as well as
$\norm{p}_{n,h,1}\ge c_1\abs{p(x_0)}$ with some $c_1>0$ uniformly in $p\in
\Pi$. Fix $R>2$. For every $\delta>0$ we can find elements $(p_l)_{l\ge
1}$ that form a $\delta$-net in $\Pi\cap \{\norm{p}_{n,h,1}\ge
c_1\max(1,c_0)(R-1)\gamma_hh^{2+\alpha}\}$ with respect to the
$L^\infty(U_h)$-norm satisfying $\norm{p_l}_{n,h,1}\asymp \delta l^{1/3}$ as
$l\to\infty$ ; for this note that, by the above norm equivalences, $\Pi\cap
\{\norm{p}_{n,h,1}\ge c_1\max(1,c_0)(R-1)\gamma_hh^{2+\alpha}\}$ with maximum
norm is isometric to $\R^3\cap\{\abs{x}\ge
c_1\max(1,c_0)(R-1)\gamma_hh^{2+\alpha}\}$ with the Euclidean metric uniformly
for $h\to 0$ and $nh\to\infty$ and use standard coverings of Euclidean balls, e.g. Lemma 2.5 in \citeasnoun{vdG06}. We obtain
\begin{align*}
&P_\theta\Big(\exists p\in \Pi: \max_{j:x_{j,n}\in U_h}\abs{Y_{j,n}-p(x_{j,n})}\le 1+\gamma_hh^{2+\alpha},\,\\
&\hspace{3cm} \max(h^{-(2+\alpha)}\abs{p(x_0)-\theta(x_0)},h^{-(1+\alpha)}\abs{p'(x_0)-\theta'(x_0)})\ge R\gamma_h\Big)\\
&=P_\theta\Big(\exists p\in \Pi: \max_{j:x_{j,n}\in U_h}\abs{\eps_{j,n}-(p(x_{j,n})-\theta(x_{j,n}))}\le 1+\gamma_hh^{2+\alpha},\\
&\hspace{3cm}\,\max(h^{-(2+\alpha)}\abs{p(x_0)-\theta(x_0)},h^{-(1+\alpha)}\abs{p'(x_0)-\theta'(x_0)})\ge R\gamma_h\Big)\\
&\le P_\theta\Big(\exists p\in \Pi: \max_{j:x_{j,n}\in U_h}\abs{\eps_{j,n}-(p(x_{j,n})-\theta_h(x_{j,n}))}\le 1+2\gamma_hh^{2+\alpha},\,\\
&\hspace{3cm} \norm{p-\theta_h}_{n,h,1}\ge \max(1,c_0)c_1(R-1)\gamma_hh^{2+\alpha}\Big)\\
&\le P_\theta\Big(\exists l\ge 1: \max_{j:x_{j,n}\in U_h}\abs{\eps_{j,n}-p_l(x_{j,n})}\le 1+2\gamma_hh^{2+\alpha}+\delta\Big)\\
&\le \sum_{l\ge 1} P_\theta\Big(\max_{j:x_{j,n}\in U_h}\abs{\eps_{j,n}-p_l(x_{j,n})}\le 1+2\gamma_hh^{2+\alpha}+\delta\Big).
%&\le \sum_{l\ge 1} \exp(-nhc(\delta+2\gamma_h) l^{1/(k+1)})\\
%&\lesssim \exp(-cnh(\delta+2\gamma_h)).
\end{align*}

From $f_\varepsilon(-1)>0$, $f_\varepsilon(+1)>0$ and the Lipschitz continuity of $f_\varepsilon$ within $[-1,1]$ we infer that any $\eps_{j,n}$ satisfies
\[\min\big(P(\eps_{j,n}\ge 1-\kappa),P(\eps_{j,n}\le -1+\kappa)\big)\ge c\kappa\]
for some constant $c>0$ and all $\kappa\in (0,1)$.
We derive an exponential inequality for any $f:U_h\to\R$ and $\Delta>0$:
\begin{align*}
&P(\max_{j:x_{j,n}\in U_h}\abs{\eps_{j,n}-f(x_{j,n})}\le 1+\Delta)\\
& \le\prod_{j:x_{j,n}\in U_h} \Big(1-\min\Big(P(\eps_{j,n}>1+\Delta-\abs{f(x_{j,n})}),
\, P(\eps_{j,n}<-1-\Delta+\abs{f(x_i)})\Big)\Big) \\
&\le \exp\Big(\sum_{j:x_{j,n}\in U_h}\log(1-c(\abs{f(x_i)}-\Delta)_+)\Big)\\
&\le \exp\Big(-c\sum_{j:x_{j,n}\in U_h}(\abs{f(x_i)}-\Delta)_+\Big)\\
&\le \exp\big(-cnh(\norm{f}_{n,h,1}-\Delta)\big),
\end{align*}
using $\log(1+h)\le h$. We therefore choose $\delta=R\gamma_hh^{2+\alpha}$ and arrive at
\begin{align*}
&P_\theta\Big(\exists p\in \Pi: p\text{ is locally admissible},\,\\
&\qquad\max(h^{-(2+\alpha)}\abs{p(x_0)-\theta(x_0)},h^{-(1+\alpha)}\abs{p'(x_0)-\theta'(x_0)})\ge R\gamma_h\Big)\\
&\le \sum_{l\ge 1} \exp\Big(-\text{const.}\cdot nh(\delta+\gamma_hh^{2+\alpha}) l^{1/3}\Big)=O\Big(\exp\Big(-\text{const.}\cdot Rnh^{3+\alpha}\Big)\Big).
\end{align*}
We conclude, substituting $h\asymp n^{-1/(3+\alpha)}$, that uniformly over $R\ge 2$
\begin{align*}
P_\theta\Big(h^{-(2+\alpha)}\abs{\hat\theta_{n,h}(x_0)-\theta(x_0)}\ge R\gamma_h\Big)
&=O(\exp(-\text{const.}\cdot R)),\\
P_\theta\Big(h^{-(1+\alpha)}\abs{\hat\theta_{n,h}'(x_0)-\theta'(x_0)}\ge R\gamma_h\Big)
&=O(\exp(-\text{const.}\cdot R)).
\end{align*}
Integrating out these exponential tail bounds yields the desired moment bound
in experiment ${\mathcal A}_n$.

All the results obtained so far remain valid for the PPP experiment ${\mathcal
B}_n$ when the empirical norm $\norm{\cdot}_{n,h,1}$ is replaced by the
rescaled $L_1(U_h)$-norm $\norm{g}_{1,U_h}:=\frac1{h}\int_{U_h}\abs{g}$, the
admissibility conditions are exchanged and  the following (easier) exponential
inequality is used:
\begin{align*}
&P_\theta\Big(X_1(\{x\in U_h,\,y>\theta(x)+f(x)-\Delta\})=0,\, X_2(\{x\in U_h,\,y<\theta(x)+f(x)+\Delta\})=0\Big)\\
&=P_0\Big(X_1(\{x\in U_h,\,y>f(x)-\Delta\})=0\Big)P_0\Big(X_2(\{x\in U_h,\,y<f(x)+\Delta\})=0\Big)\\
& = \exp\Big(-nf_\eps(1)\int_{U_h} (f(x)-\Delta)_+f_D(x)\,dx\Big) \\
& \hspace{3cm}\cdot \exp\Big(-nf_\eps(-1)\int_{U_h} (-f(x)-\Delta)_+f_D(x)\,dx\Big)\\
&\le\exp\big(-c'nh(\norm{f}_{1,U_h}-\Delta)\big)
\end{align*}
with some constant $c'>0$.
\hfill $\square$\\

\section{Design adjustment for the regression experiment} \label{3}

We use a piecewise constant approximation strategy and introduce the intervals
\begin{equation}
I_{k,n} = [k/m,(k+1)/m), \quad k=0,\ldots,m-2, \text{ and }I_{m-1,n} = [(m-1)/m,1]
\end{equation}
for some integer $m$. For any design point $x_{j,n}\in I_{k,n}$ we introduce
the centre of the interval
\begin{equation}
\xi_{j,n}:=(k+1/2)/m \text{ for } x_{j,n}\in I_{k,n}.
\end{equation}
Now we apply a sample splitting scheme and write $J_n$ for the collection of odd
$j\in \{1,\ldots,n\}$. The experiment ${{\mathcal A}_n}$ is considered as the
totality of the two independent data sets ${\bf X}=(Y_{j+1,n})_{j\in J_n}$ and
${\bf Y}' = (Y_{j,n})_{j\in J_n}$.

Subsequently, we shall not touch upon $\bf X$ to establish asymptotic
equivalence, but just assume the existence of sufficiently good estimators
based on the data $\bf X$. Therefore, we forget about the specific definition
of ${\bf X}$ and write ${\bf X}^*$ instead.

\begin{defn}
Let ${\bf X}^*$ be an arbitrary observation in a Polish space, which is independent of
${\bf Y}'$. We generalize the experiment ${{\mathcal A}_n}$ to ${{\mathcal
A}_n}^*$, which consists of the data ${\bf Y}'$ and ${\bf X}^*$.
\end{defn}

The original experiment ${{\mathcal A}_n}$ is still included by putting ${\bf X}^*
= {\bf X}$. This enables us to repeatedly use the following results later also when
${\bf X}^*$ will denote a PPP observation.

In a first step we show asymptotic equivalence for the regression experiment
${{\mathcal A}_n}^*$ with the same experiment, but where for $j\in J_n$ the
regression function is observed at the interval centres $\xi_{j,n}$.

\begin{defn}\label{defC}
In experiment ${\mathcal C}_n$ we observe independently the vectors ${\bf X}^*$
as under experiment ${{\mathcal A}_n}^*$ and, independently, the vector ${\bf
Z}$ with the components
$$ Z_{j,n} \, = \, \theta(\xi_{j,n}) + \varepsilon_{j,n}\,, \qquad j \in J_n\,. $$
\end{defn}

\begin{lem} \label{L:1}
Choose $m\in\N$ such that $m^{-1}=o(n^{-1/2})$ holds and assume that an
estimator $\hat{\theta}'$ can be constructed based on the data set ${\bf X}^*$
with
$$ \sup_{\theta \in \Theta}\,  \sup_{x\in [0,1]} \, E_\theta|\hat{\theta}'(x) - \theta'(x)| \, = \, o(mn^{-1}). $$
\noindent Then the experiments ${{\mathcal A}_n}^*$ and ${{\mathcal C}_n}$ are
asymptotically equivalent.
\end{lem}

\noindent {\it Proof of Lemma \ref{L:1}:} The observations ${\bf Y}'$ from the
experiment ${{\mathcal A}_n}^*$ are transformed into the data set $\tilde{{\bf
Y}}$ with the components
\begin{align*}
\tilde{Y}_{j,n} & \, = \, Y_{j,n} - \hat{\theta}'(\xi_{j,n}) (x_{j,n} - \xi_{j,n}) \\
& \, = \, \theta(x_{j,n}) - \theta'(\xi_{j,n}) (x_{j,n}-\xi_{j,n}) - [\hat{\theta}'(\xi_{j,n}) - \theta'(\xi_{j,n})] (x_{j,n}-\xi_{j,n}) + \varepsilon_{j,n}\,,
\end{align*}
\noindent for all $j\in J_n$. The data set ${\bf X}^*$ is not affected by this
transformation. As $\hat{\theta}'$ is based on the data ${\bf X}^*$, this
transformation is invertible so that the original data are uniquely
reconstructable from the transformed ones; and observing $({\bf X}^*,{\bf Y}')$
on the one hand and $({\bf X}^*,\tilde{{\bf Y}})$ on the other hand is
equivalent. Therefore, for any measurable functional $R$ with $\|R\|_\infty
\leq 1$ we observe that
\begin{align}
\big|E_\theta R({\bf X}^*,\tilde{{\bf Y}}) - E_\theta R({\bf X}^*,{\bf Z})\big| & \, \leq \, E_\theta \big|E_\theta \{R({\bf X}^*,\tilde{{\bf Y}}) | {\bf X}^*\} - E_\theta \{R({\bf X}^*,{\bf Z}) | {\bf X}^*\}\big| \label{eq:A1.0}\\
& \, \leq \, \sum_{j\in J_n} E_\theta \|f_{\tilde{Y}_{j,n}|X^*} - f_{Z_{j,n}|X^*}\|_1\,,\label{eq:CondTV}
\end{align}
\noindent where $\|\cdot\|_1$ denotes the $L_1(\R)$-norm; in general, $f_{Y|X}$
stands for the conditional density of $Y$ given $X$. The conditional
independence of the $\tilde{Y}_{j,n}$ and the $Z_{j,n}$ given ${\bf X}^*$ as well
as an elementary telescopic sum argument with respect to the
$L_1(\R)$-distance of the multivariate conditional densities of $\tilde{{\bf
Y}}$ and ${\bf Z}$ given ${\bf X}^*$ have been exploited. We obtain by the Lipschitz continuity of $\phi$
\begin{align} \label{eq:A1.1}
& \|f_{\tilde{Y}_{j,n}|X^*} - f_{Z_{j,n}|X^*}\|_1 \, \leq \, 2 \|\varphi\|_\infty \cdot |\Delta_{1,j,n}| \, + \, \int_{-1}^1 |\varphi(x + \Delta_{1,j,n}) - \varphi(x)| dx
\leq  4C_\eps \cdot |\Delta_{1,j,n}| \,,
\end{align}
%\begin{align} \label{eq:A1.1}
%& \|f_{\tilde{Y}_{j,n}|X} - f_{Z_{j,n}|X}\|_1 \, \leq \, 2 \|\varphi\|_\infty \cdot |\Delta_{1,j,n}| \, %+ \, \int |\varphi(x + \Delta_{1,j,n}) - \varphi(x)| dx\,, \end{align}
\noindent where
$$ \Delta_{1,j,n} \, = \, \theta(x_{j,n}) - \theta(\xi_{j,n}) - \theta'(\xi_{j,n}) (x_{j,n}-\xi_{j,n}) - [\hat{\theta}'(\xi_{j,n}) - \theta'(\xi_{j,n})] (x_{j,n}-\xi_{j,n})\,. $$
We conclude that the total variation distance between $({\bf X}^*,\tilde{{\bf Y}})$ and $({\bf X}^*,{\bf Z})$ is bounded from above by
$$ \mbox{const.}\cdot \sum_{j\in J_n} E_\theta \big(|\Delta_{1,j,n}|\big)\,. $$
\noindent By the H\"older constraints imposed on the parameter class
$\Theta$ we derive that
$$ |\Delta_{1,j,n}| \, \leq \, \mbox{const.}\cdot\big( m^{-2} + |\hat{\theta}'(x_{j,n}) - \theta'(x_{j,n})| m^{-1} \big)\,. $$
\noindent Using $m^{-2}=o(n^{-1})$ and the convergence rate of
$\hat{\theta}'$, we conclude that the Le Cam distance
between the experiments ${{\mathcal A}_n}^*$ and ${{\mathcal C}_n}$ tends to zero uniformly in $\theta$, which gives the assertion of the lemma. \hfill $\square$ \\

Usually, the bound on the total variation of product measures which is used in
the proof is suboptimal, but here the order is optimal due to the singular parts in the measures. Note also that the data $Z_{j,n}$ may be viewed as
random responses drawn from a regression function which is locally constant on
the intervals $I_{k,n}$ with the values $\theta(\xi_{j,n})$ when $x_{j,n} \in
I_{k,n}$.

\section{Asymptotic equivalence for step functions} \label{4}

We revisit the experiment ${\mathcal C}_n$ from Definition \ref{defC}. The data
$Z_{j,n}$ may be transformed into
$$\tilde{Z}_{j,n} \, = \, Z_{j,n} - \hat{\theta}(\xi_{j,n}),$$
\noindent where $\hat{\theta}$ denotes a preliminary estimator of $\theta$
which is based on the data from ${\bf X}^*$ as contained in the experiment
${{\mathcal C}_n}$. Again this transformation is invertible so that the
experiment ${\mathcal C}_n$ is equivalent to the experiment ${{\mathcal C}_n}'$
under which one observes the data ${\bf X}^*$ and the vector $\tilde{{\bf
Z}}=(\tilde{Z}_{j,n})_{j\in J_n}$. The $\tilde{Z}_{j,n}$, $j\in J_n$, are
conditionally independent given ${\bf X}^*$ and have the conditional densities
\begin{equation}\label{eq:Delta0jn}
f_\varepsilon(x - \Delta_{0,j,n})=\varphi(x-\Delta_{0,j,n})
1_{[\Delta_{0,j,n}-1,\Delta_{0,j,n}+1]}(x) \text{ with }\Delta_{0,j,n}
=\theta(\xi_{j,n}) - \hat{\theta}(\xi_{j,n}).
\end{equation}

The next key step is to replace these densities by those with unshifted
$\varphi$ where local minima and maxima will turn out to be sufficient
statistics.

\begin{defn}\label{defD}
Let $W_{j,n}$, $j\in J_n$, conditionally on ${\bf X}^*$ be independent random
variables with respective densities
\begin{align*} f_{W,j}(x) & \, = \, \varphi(x) \Big(\int_{\Delta_{0,j,n}-1}^{\Delta_{0,j,n}+1} \varphi(t) dt\Big)^{-1} 1_{[\Delta_{0,j,n}-1,\Delta_{0,j,n}+1]}(x)\,,\qquad j\in J_n\,,
\end{align*}
where $\Delta_{0,j,n}$ is given in \eqref{eq:Delta0jn}. The experiment in which
${\bf X}^*$ and the $W_{j,n}$, $j\in J_n$, are observed for $\theta\in\Theta$
is denoted by ${\mathcal  D}_n$.
\end{defn}

\begin{lem} \label{L:2}
Suppose that an estimator $\hat{\theta}$ of $\theta$ can be constructed based
on the data set ${\bf X}^*$ such that
\begin{equation} \label{eq:est}
\sup_{\theta \in \Theta}\,  \sup_{x\in [0,1]} \, E_\theta|\hat{\theta}(x) -
\theta(x)|^2 \, = \, O(n^{-1-\delta})\,,
\end{equation}
\noindent for some
$\delta>0$. Then the experiments ${\mathcal C}_n$ and ${\mathcal D}_n$ are
asymptotically equivalent.
\end{lem}

\noindent {\it Proof of Lemma \ref{L:2}:} By Le Cam's inequality  and the subadditivity of the squared Hellinger distance $H$ for product measures (cf. Section 2.4 in \citeasnoun{T09} or Appendix 9.1 in \citeasnoun{R08}) we deduce that for any measurable functional $R$ with
$\|R\|_\infty\leq 1$ we have
\begin{align} \nonumber
|E_\theta R({\bf X}^*,\tilde{{\bf Z}}) - E_\theta R({\bf X}^*,{\bf W})| \, &\leq \, E_\theta \idotsint \Big|\prod_{j\in J_n} f_\varepsilon(y_j - \Delta_{0,j,n}) - \prod_{j\in J_n} f_{W,j}(y_j)\Big| d{\bf y} \\
& \leq \,
%2 \Big\{1 - \exp \Big( -\mbox{const.}\cdot \sum_{j\in J_n} E_\theta H^2\big(f_{W,j}, f_\varepsilon(\cdot - \Delta_{0,j,n})\big)\Big)\Big\}^{1/2} \\
 2 \sum_{j\in J_n} E_\theta H^2\big(f_{W,j}, f_\varepsilon(\cdot - \Delta_{0,j,n})\big)\,,\label{eq:tool}
\end{align}
\noindent where the expectation is taken over $\Delta_{0,j,n}$. Hence, it remains to be shown that the %expression
%$$ \sum_{j\in J_n} E_\theta H^2\big(f_{W,j}, f_\varepsilon(\cdot - \Delta_{0,j,n})\big)\,, $$
%\noindent
sum converges to zero uniformly with respect to $\theta\in\Theta$. That sum equals
\begin{align*}
\sum_{j\in J_n} & E_\theta\int_{\Delta_{0,j,n}-1}^{\Delta_{0,j,n}+1}
\left(\sqrt{\varphi(x)} \Big(\int_{\Delta_{0,j,n}-1}^{\Delta_{0,j,n}+1} \varphi(t) dt\Big)^{-1/2}
- \sqrt{\varphi(x-\Delta_{0,j,n})}\right)^2\,dx \\
& \, \leq \, 4 C_\eps^2  \, \big(2\, + \, \{\inf_{|x|\leq 1} \varphi(x)\}^{-1}\big) \, \sum_{j\in J_n} E_\theta \Delta_{0,j,n}^2\,,
\end{align*}

\noindent since $\varphi$ is strictly positive, continuous and satisfies the condition (\ref{eq:error0}).
The imposed convergence rate of the estimator $\hat{\theta}$ yields that the supremum taken over
$\theta\in\Theta$ tends to zero at the rate $O(n^{-\delta})$ and the proof is complete. \hfill $\square$ \\

The conditional joint density of the $W_{j,n}$, $j\in J_n$, given ${\bf X}^*$
from the experiment ${{\mathcal D}_n}$ can be represented by
\begin{align}
 & f_{W}({\bf w}) \, = \, \prod_{j\in J_n} f_{W,j}(w_j) \, = \,
 \Big(\prod_{j\in J_n} \varphi(w_j)\Big) \Big(\prod_{j\in J_n}
 \int_{\Delta_{0,j,n}-1}^{\Delta_{0,j,n}+1} \varphi(t) dt\Big)^{-1} \label{eq:A0.4}  \\
 & \cdot \Big(\prod_{k=0}^{m-1} 1(\min\{w_j : x_{j,n}\in I_{k,n}\}\ge \Delta_{0,j(k),n}-1)\cdot
 1(\max\{w_j : x_{j,n}\in I_{k,n}\}\le \Delta_{0,j(k),n}+1)\Big)\,,\nonumber
\end{align}
\noindent where the $I_{k,n}$ are as in Section \ref{3} and $j(k) = \min\{l \in J_n : x_{l,n} \in I_{k,n}\}$, ${\bf w}=(w_j)_{j\in J_n}$. Note that the parameter $\theta$ is included in the term $\Delta_{0,j(k),n}$.

\begin{defn}\label{defE}
In experiment ${\mathcal E}_n$ only the data $({\bf X}^*,s_{k,n},S_{k,n})$,
$k=0,\ldots,m-1$, with
\begin{align*}
s_{k,n} & \, = \, \min\{W_{j,n} : x_{j,n}\in I_{k,n}\}\,, \\
S_{k,n} & \, = \, \max\{W_{j,n} : x_{j,n}\in I_{k,n}\}\,,
\end{align*}
are observed for $\theta\in\Theta$.
\end{defn}

An inspection of (\ref{eq:A0.4}) yields that $({\bf X}^*,s_{k,n},S_{k,n})$, $k=0,\ldots,m-1$,
provides a sufficient statistic for the whole empirical information
contained in $({\bf X}^*,\{W_{j,n}: j\in J_n\})$ by the Fisher-Neyman factorization theorem.

Sufficiency implies equivalence (e.g. Lemma 3.2 in \citeasnoun{BL96}) and we have
\begin{lem}
Experiments ${\mathcal D}_n$ and ${\mathcal E}_n$ are equivalent.
\end{lem}

In the following we study the conditional distribution of $(s_{k,n},S_{k,n})$ given ${\bf X}^*$. Note that, conditionally on ${\bf X}^*$, the $(s_{k,n},S_{k,n})$ are independent for $k=0,\ldots,m-1$ as the intervals $I_{k,n}$ are disjoint. We derive that
\begin{align*}
P[s_{k,n} > x,\, S_{k,n} \leq y|{\bf X}^*] & \, = \, P[W_{j,n} \in (x,y], \, \forall j\in J_n \mbox{ with }x_{j,n} \in I_{k,n} | {\bf X}^*] \\
& \, = \, \Big(\int_x^y f_{W,j(k)}(t) dt\Big)^{l_{k,n}}\,,
\end{align*}
\noindent for $y>x$. Thus we obtain the conditional joint density of $(s_{k,n},S_{k,n})$ via
\begin{align*}
f_{(s_{k,n},S_{k,n})}(x,y) & \, = \, - \frac{\partial^2}{\partial x \partial y} P[s_{k,n} > x,\, S_{k,n} \leq y| {\bf X}^*] \\
& \, = \,  A_{k,n}(x,y) \cdot l_{k,n}(l_{k,n}-1) f_{W,j(k)}(x) f_{W,j(k)}(y) 1_{\{y\geq x\}}\,,
\end{align*}
\noindent where
\begin{align*}
A_{k,n}(x,y) & \, = \, \Big(1 - \int_{\Delta_{0,j(k),n}-1}^{x} f_{W,j(k)}(t) dt - \int_{y}^{\Delta_{0,j(k),n}+1} f_{W,j(k)}(t) dt\Big)^{l_{k,n}-2}\,.
\end{align*}

\begin{defn}
Consider for each $k$ two conditionally on ${\bf X}^*$ independent random variables $s'_{k,n}$ and $S'_{k,n}$ with conditional exponential densities
\begin{align*}
f_{s'_{k,n}}(x) & \, = \, (l_{k,n}-2) f_{W,j(k)}(\Delta_{0,j(k),n}-1) \, \exp\big(- (l_{k,n}-2) f_{W,j(k)}(\Delta_{0,j(k),n}-1)\\ & \hspace{5cm} \cdot (x-\Delta_{0,j(k),n}+1)\big)\, {\bf 1}_{[\Delta_{0,j(k),n}-1,\infty)}(x) \,, \\
f_{S'_{k,n}}(x) & \, = \, (l_{k,n}-2) f_{W,j(k)}(\Delta_{0,j(k),n}+1) \, \exp\big(- (l_{k,n}-2) f_{W,j(k)}(\Delta_{0,j(k),n}+1) \\ & \hspace{5cm} \cdot (-x+\Delta_{0,j(k),n}+1)\big) \, {\bf 1}_{(-\infty,\Delta_{0,j(k),n}+1]}(x)\,,
\end{align*}
\noindent and the joint density $f_{(s'_{k,n},S'_{k,n})}$. Then the experiment
${\mathcal F}_n$ is obtained by observing ${\bf X}^*$ as well as conditionally
on ${\bf X}^*$ independent tuples $(s'_{k,n},S'_{k,n})$, $k=0,\ldots, m-1$.
\end{defn}

%The following lemma gives us asymptotic proximity between $f_{(s'_{k,n},S'_{k,n})}$ and $f_{(s_{k,n},S_{k,n})}$.

\begin{lem} \label{L:10}
%We select $m$ as in Lemma \ref{L:1}.
Assume that $m \leq \mbox{const.}\cdot n^{1-\delta}$ for some $\delta>0$ and that
\begin{equation} \label{eq:bound}
\sup_{k=0,\ldots,m-1} |\Delta_{0,j(k),n}| \leq 2 C_\Theta\,, \qquad \mbox{a.s.}\,, \quad \forall \theta\in\Theta\,.
\end{equation}

Conditionally on the data set ${\bf X}^*$, the squared Hellinger distance between  $f_{(s'_{k,n},S'_{k,n})}$ and $f_{(s_{k,n},S_{k,n})}$ satisfies
$$ H^2(f_{(s'_{k,n},S'_{k,n})},f_{(s_{k,n},S_{k,n})}) \, \leq \, \mbox{const.}\cdot
%(\log^4 n) n^{-2/3-\delta/2}
\{\log(n/m)\}^4 (m/n)^2\,,$$
\noindent where const. is uniform with respect to $n$, ${\bf X}^*$, $\theta$ and $k$.
\end{lem}

\begin{rem}
This approximation result together with the ensuing corollary tells us that we need to choose the number $m$ of intervals of polynomially smaller order than $n^{2/3}$. To see that we cannot hope for a better approximation order, note that already in the most simple univariate case where $s:=\min(U_i, i=1,\ldots, I)$ with $U_i$ i.i.d. uniform on $[0,1]$ and $s'$ exponentially distributed with intensity $I\in\N$, we have
for $I\to\infty$
\begin{align*}
H^2(f_s,f_{s'})&\ge\int_0^{1/I}\big( \sqrt{I(1-x)^{I-1}}-\sqrt{I\exp(-Ix)}\big)^2dx\\
&\approx \big((1-1/I)^{(I-1)/2}-\exp(-1/2)\big)^2\asymp I^{-2}.
\end{align*}
\end{rem}

\begin{cor}\label{cor1}
We assume that an estimator $\hat{\theta}$ of $\theta$ can be constructed from
the data ${\bf X}^*$ such that (\ref{eq:bound}) holds. For $m=O(n^{2/3-\delta})$ with some $\delta>0$ as $n\to\infty$ the
experiments ${\mathcal E}_n$ and ${\mathcal F}_n$ are asymptotically
equivalent.
\end{cor}

\noindent {\it Proof of Corollary \ref{cor1}:} Focussing on the total variation
distance between the distributions of the data $\big({\bf
X}^*,\{(s'_{k,n},S'_{k,n}):k=0,\ldots,m-1\}\big)$ and $\big({\bf
X}^*,\{(s_{k,n},S_{k,n}):k=0,\ldots,m-1\}\big)$ we consider for any measurable
functional $R$ on an appropriate domain and $\|R\|_\infty \leq 1$ that
\begin{align*}
& \big| E_\theta R({\bf X}^*,s_{0,n},S_{0,n},\ldots,s_{m-1,n},S_{m-1,n}) - E_\theta R({\bf X}^*,s'_{0,n},S'_{0,n},\ldots,s'_{m-1,n},S'_{m-1,n})\big| \\
%& \, \leq \, 2 E_\theta \Big\{1 - \prod_{k=0}^{m-1} \Big(1 - \frac12 %H^2(f_{(s_{k,n},S_{k,n})},f_{(s'_{k,n},S'_{k,n})})\Big)\Big\}^{1/2} \\
& \, \le \, 2 \sum_{k=0}^{m-1}E_\theta H^2(f_{(s_{k,n},S_{k,n})},f_{(s'_{k,n},S'_{k,n})})\\
&  \, \leq \, \mbox{const.}\cdot n^{-\delta/2} \log^2 n\,,
\end{align*}
\noindent using the conditional independence of the $(s_{k,n},S_{k,n})$, $k=0,\ldots,m-1$, on the one hand and the $(s'_{k,n},S'_{k,n})$, $k=0,\ldots,m-1$, on the other hand and arguments as in the proof of Lemma \ref{L:2}; as well as Lemma \ref{L:10} in the last line.  Thus the total variation distance between the distributions of the data $\big({\bf X}^*,\{(s'_{k,n},S'_{k,n}):k=0,\ldots,m-1\}\big)$ and $\big({\bf X}^*,\{(s_{k,n},S_{k,n}):k=0,\ldots,m-1\}\big)$ converges to zero as $n\to\infty$, which proves the claim of the corollary.
\hfill $\square$ \\

\noindent {\it Proof of Lemma \ref{L:10}:} First we mention that, although the arguments of the Hellinger distance are most usually densities, its definition $H^2(f,g) \, = \, \int (\sqrt{f}(x)-\sqrt{g}(x))^2 dx$ may easily be extended to all nonnegative functions $f,g \in L_1(\R)$. This fact will be used in the sequel. Moreover, note that $l_{k,n}\asymp n/m\ge \text{const.}\cdot n^\delta$ holds uniformly over $k$ by our design assumption (\ref{eq:con_X}). We set
$$ f_{1,k,n}(x,y) = \frac{(l_{k,n}-2)^2}{l_{k,n}(l_{k,n}-1)} f_{(s_{k,n},S_{k,n})}(x,y)\,, $$
\noindent so that
\begin{equation} \label{eq:A5.0} H^2(f_{1,k,n},f_{(s_{k,n},S_{k,n})}) \, \leq \, \frac{(4-3l_{k,n})^2}{l_{k,n}(l_{k,n}-1)(l_{k,n}-2)^2} \, \asymp \,
%n^{-2/3 - \delta/2}
l_{k,n}^{-2}\,, \end{equation}
%\noindent since the selection of $m$ implies that
%$$\max_{k=0,\ldots,m-1} l_{k,n} \asymp \min_{k=0,\ldots,m-1} l_{k,n} \asymp n^{1/3+\delta/4}\,,$$ due to condition (\ref{eq:con_X}).
Note that the support of $f_{(s_{k,n},S_{k,n})}$ and hence of $f_{1,k,n}$ is included in the square $Q_{k,n} = [\Delta_{0,j(k),n}-1,\Delta_{0,j(k),n}+1]^{2}$. A sub-square is defined by
$$ Q_{1,k,n} = [\Delta_{0,j(k),n}-1,\Delta_{0,j(k),n}-1+a_{k,n}] \times [\Delta_{0,j(k),n}+1-a_{k,n},\Delta_{0,j(k),n}+1] \subseteq Q_{k,n}\,,$$
\noindent which will contain most probability masses, and we set $Q_{2,k,n} =
Q_{k,n} \backslash Q_{1,k,n}$ where $a_{k,n} = d_0 l_{k,n}^{-1} \log l_{k,n}$
with a constant $d_0>0$ for $n$ sufficiently large. We split the Hellinger
distance into integrals over disjoint domains so that

\begin{align} \nonumber
H^2(f_{1,k,n},f_{(s'_{k,n},S'_{k,n})}) & \, \leq \, \int_{Q_{1,k,n}} \big(\sqrt{f_{1,k,n}}(x,y) - \sqrt{f_{(s'_{k,n},S'_{k,n})}}(x,y)\big)^2 dx\, dy  \\ \nonumber
& \quad + \, 2 \int_{Q_{2,k,n}} f_{1,k,n}(x,y) dx\, dy \, + \, 2 P[s'_{k,n} > \Delta_{0,j(k),n}-1+a_{k,n} | {\bf X}^*] \\ & \nonumber \qquad + \, 2 P[S'_{k,n} < \Delta_{0,j(k),n}+1-a_{k,n} | {\bf X}^*]\\ \label{eq:A4.0}
&=:T_1+T_2+T_3+T_4\,.
\end{align}
The conditions (\ref{eq:error0}) and (\ref{eq:bound}) combined with the
positivity of $\varphi$ imply that $\|f_{W,j(k)}\|_\infty \leq C_\eps$ and that
\begin{align*}
\int_{\Delta_{0,j(k),n}-1}^{x} f_{W,j(k)}(t) dt &  \geq \mbox{const.}\cdot (x-\Delta_{0,j(k),n}+1)\,, \, \forall x \in [\Delta_{0,j(k),n}-1,\Delta_{0,j(k),n}+1], \\
\int_y^{\Delta_{0,j(k),n}+1} f_{W,j(k)}(t) dt &  \geq \mbox{const.}\cdot (\Delta_{0,j(k),n}+1-y)\,, \, \forall y \in [\Delta_{0,j(k),n}-1,\Delta_{0,j(k),n}+1].
\end{align*}
\noindent As the Lebesgue measure of $Q_{k,n}$ is equal to $4$, thus bounded,
we deduce by the definition of $f_{1,k,n}$ and $f_{(s_{k,n},S_{k,n})}$ that
\[T_2\, \leq \, c_\nu n^{-\nu}\,, \]
for each $\nu>0$ when selecting the constant $d_0$ in the definition of
$a_{k,n}$ sufficiently large where $c_\nu$ denotes a finite constant which
depends on neither the data ${\bf X}^*$, $\theta$ nor $x,y$.

Concerning terms $T_3$ and $T_4$, easy calculations yield that these terms are
equal to $2\exp\big\{-a_{k,n} (l_{k,n}-2)  f_{W,j(k)}(\Delta_{0,j(k),n}\mp
1)\big\}$, respectively. We may use (\ref{eq:error0}), (\ref{eq:bound}) and
$\varphi>0$ to show that $f_{W,j(k)}(\Delta_{0,j(k),n}\mp 1) \geq
\mbox{const}$. Again choosing the constant $d_0$ sufficiently large implies
that $\max\{T_3,T_4\}\, \leq \, c'_\nu n^{-\nu}$, for any $\nu>0$ with a
constant $c'_\nu$ which has the same properties as $c_\nu$.

Let us focus on the main term $T_1$. For $(x,y) \in Q_{1,k,n}$, we have
\begin{align*}
\log A_{k,n}(x,y) & \, = \, (l_{k,n}-2) \, \Big(-\int_{\Delta_{0,j(k),n}-1}^x f_{W,j(k)}(t) dt - \int_y^{\Delta_{0,j(k),n}+1} f_{W,j(k)}(t) dt\Big) \\ & \hspace{8.5cm} + R_{1,k,n}(x,y)\,,
\end{align*}
\noindent where $\sup_{(x,y) \in Q_{1,k,n}} \max_{k=0,\ldots,m-1} |R_{1,k,n}(x,y)| \leq \mbox{const.}\cdot l_{k,n} a_{k,n}^2 \asymp
%(\log^2 n) n^{-1/3-\delta/4}
l_{k,n}^{-1}\log^2 l_{k,n}$ by the Taylor expansion of the logarithm. Furthermore, the functions to be integrated are locally approximated by constant functions,
\begin{align*}
& - \int_{\Delta_{0,j(k),n}-1}^x f_{W,j(k)}(t) dt - \int_y^{\Delta_{0,j(k),n}+1} f_{W,j(k)}(t) dt \\
& \, = \, - f_{W,j(k)}(\Delta_{0,j(k),n} - 1) \cdot (x - \Delta_{0,j(k),n} + 1) \\ & \hspace{4.5cm} - f_{W,j(k)}(\Delta_{0,j(k),n} + 1) \cdot (-y + \Delta_{0,j(k),n} + 1) + R_{2,k,n}(x,y)\,,
\end{align*}
\noindent where $\sup_{(x,y) \in Q_{1,k,n}} \max_{k=0,\ldots,m-1} |R_{2,k,n}(x,y)| \leq \mbox{const.}\cdot
%(\log^2 n) n^{-2/3-\delta/2}
l_{k,n}^{-2}\log^2 l_{k,n}$, using the Lip\-schitz continuity of $\phi$.
%that $\sup_{t\in[\Delta_{0,j(k),n} - 1, \Delta_{0,j(k),n} + 1]} |f_{W,j(k)}'(t)| \leq C_\eps/2$ %where $f_{W,j(k)}'(t)$ at $t=\Delta_{0,j(k),n}-1$ and $t=\Delta_{0,j(k),n}+1$ is to be understood as %the right and left side derivative, respectively.

We introduce $B_{k,n}(x,y) := A_{k,n}(x,y) f_{W,j(k)}(x) f_{W,j(k)}(y) (l_{k,n}-2)^2$ so that $B_{k,n}(x,y)$ coincides with $f_{1,k,n}(x,y)$ on its restriction to $(x,y) \in Q_{1,k,n}$ for $n$ large enough, as well as
$$ \tilde{B}_{k,n}(x,y)\, := \, f_{(s'_{k,n},S'_{k,n})}(x,y) \frac{f_{W,j(k)}(x) f_{W,j(k)}(y)}{f_{W,j(k)}(\Delta_{0,j(k),n}-1) f_{W,j(k)}(\Delta_{0,j(k),n}+1)}\,. $$
We obtain
\begin{align*}
B_{k,n}^{1/2}(x,y) & \, = \, \tilde{B}_{k,n}^{1/2}(x,y) \, \exp\big(R_{1,k,n}(x,y)/2 + (l_{k,n}-2) R_{2,k,n}(x,y)/2\big) \\
& \, = \, \tilde{B}_{k,n}^{1/2}(x,y) + \tilde{B}_{k,n}^{1/2}(x,y) R_{3,k,n}(x,y)\,,
\end{align*}
\noindent where $\sup_{(x,y) \in Q_{1,k,n}} \max_{k=0,\ldots,m-1} |R_{3,k,n}(x,y)| \leq \mbox{const.}\cdot %(\log^2 n)\, n^{-1/3-\delta/4}
l_{k,n}^{-1}\log^2 l_{k,n}$ so that
$$ \tilde{B}_{k,n}^{1/2}(x,y) \, = \,  f_{(s'_{k,n},S'_{k,n})}^{1/2}(x,y) + f_{(s'_{k,n},S'_{k,n})}^{1/2}(x,y) R_{4,k,n}(x,y)\,, $$
\noindent where
\begin{align*}
|R_{4,k,n}(x,y)| & \, \leq \, \mbox{const.}\cdot \big(|f_{W,j(k)}(x) - f_{W,j(k)}(\Delta_{0,j(k),n}-1)|  \\ & \hspace{4cm}+  |f_{W,j(k)}(y) - f_{W,j(k)}(\Delta_{0,j(k),n}+1)|\big) \\
& \, \leq \,  \mbox{const.}\cdot a_{k,n} \, \asymp \, %(\log n) n^{-1/3-\delta/4}
l_{k,n}^{-1}\log l_{k,n}\,,
\end{align*}
\noindent where the conditions (\ref{eq:bound}), (\ref{eq:error0}) and their consequences have been used. We conclude that
\begin{align*}
B_{k,n}^{1/2}(x,y) & \, = \, f_{(s'_{k,n},S'_{k,n})}^{1/2}(x,y) + f_{(s'_{k,n},S'_{k,n})}^{1/2}(x,y) R_{5,k,n}(x,y)\,, \end{align*}
\noindent where $\sup_{(x,y) \in Q_{1,k,n}} \max_{k=0,\ldots,m-1} |R_{5,k,n}(x,y)| \leq \mbox{const.}\cdot %(\log^2 n) n^{-1/3-\delta/4}
l_{k,n}^{-1}\log^2 l_{k,n}$. Hence, the term $T_1$ is bounded from above by
\begin{align*}
T_1 & \, \leq \, \int_{Q_{1,k,n}} R_{5,k,n}^2(x,y) f_{(s'_{k,n},S'_{k,n})}(x,y) \, dx \, dy \, \leq \, \mbox{const.}\cdot %(\log^4 n) n^{-2/3-\delta/2}
(\log^4 l_{k,n}) l_{k,n}^{-2}\,,
\end{align*}
\noindent as the density $f_{(s'_{k,n},S'_{k,n})}$ integrates to one. By inserting the upper bounds on $T_1,\ldots,T_4$ into (\ref{eq:A4.0}) and combining that result with (\ref{eq:A5.0}), we complete the proof. \hfill $\square$ \\

\begin{defn}
In experiment ${\mathcal G}_n$ we observe the data $({\bf
X}^*,(d_{k,n},D_{k,n})_{k=0,\ldots,m-1})$ for $\theta\in\Theta$ where
$d_{0,n},D_{0,n},\ldots,d_{m-1,n},D_{m-1,n}$ are independent random variables,
also independent of ${\bf X}^*$, with densities
\begin{align*}
f_{d_{k,n}}(x) & \, = \, \rho_{k,n} \varphi(-1) \, \exp\big(- \rho_{k,n} \varphi(-1) \cdot [x-\theta(\xi_{j(k),n})]\big)\, {\bf 1}_{[\theta(\xi_{j(k),n}),\infty)}(x) \,, \\
f_{D_{k,n}}(x) & \, = \, \rho_{k,n} \varphi(1) \, \exp\big( \rho_{k,n} \varphi(1) \cdot [x-\theta(\xi_{j(k),n})]\big) \, {\bf 1}_{(-\infty,\theta(\xi_{j(k),n})]}(x)\,,
\end{align*}
\noindent where $\rho_{k,n} = (n/2) \int_{I_{k,n}} f_D(t) dt$ with $f_D$ as in (\ref{eq:con_X_0}).
\end{defn}

\begin{lem} \label{L:11}
We select $m$ such that $m=o(n^{2/3})$. Also we assume the existence of an estimator $\hat{\theta}$ of $\theta$ based on ${\bf X}^*$ such that (\ref{eq:bound}) and
\begin{equation*}
\sup_{\theta \in \Theta}\,  \sup_{x\in [0,1]} \, E_\theta|\hat{\theta}(x) -
\theta(x)|^2 \, = \, o(m^{-1}) \end{equation*} \noindent are fulfilled. Then
the experiments ${\mathcal F}_n$ and ${\mathcal G}_n$ are asymptotically
equivalent as $n\to\infty$.
\end{lem}

\noindent {\it Proof of Lemma \ref{L:11}:} As the estimator $\hat{\theta}$ is
based on the data set ${\bf X}^*$ the transformation ${\mathcal T}$ which maps
the observations $\big({\bf X}^*,\{(s'_{k,n},S'_{k,n}):k=0,\ldots,m-1\}\big)$
to $\big({\bf X}^*,\{(s''_{k,n},S''_{k,n}):k=0,\ldots,m-1\}\big)$ with
$s''_{k,n} = s'_{k,n}+\hat{\theta}(\xi_{j(k),n})+1$ and $S''_{k,n} =
S'_{k,n}+\hat{\theta}(\xi_{j(k),n})-1$ is invertible. Therefore, the experiment
under which the data $\big({\bf
X}^*,\{(s''_{k,n},S''_{k,n}):k=0,\ldots,m-1\}\big)$ are observed is equivalent
to the experiment ${\mathcal F}_n$.

The squared Hellinger distance between the exponential densities with the same endpoint and the scaling parameters $\mu_1$ and $\mu_2$ turns out to be
$2 (\mu_1-\mu_2)^2 (\mu_1+\mu_2)^{-1} (\sqrt{\mu_1} + \sqrt{\mu_2})^{-2}$.

Also, (\ref{eq:con_X_0}) implies that $|l_{k,n} - \rho_{k,n}| \leq 2$ for all
$k=0,\ldots,m-1$. We may set $\mu_{1,\pm} = \rho_{k,n} \varphi(\pm 1)
\int_{\Delta_{0,j(k),n}-1}^{\Delta_{0,j(k),n}+1} \varphi(t) dt$ and
$\mu_{2,\pm} = (l_{k,n}-2) \varphi(\Delta_{0,j(k),n}\pm 1)$. Hence,
\begin{align*}
& H^2(f_{S''_{k,n}},f_{D_{k,n}}) + H^2(f_{s''_{k,n}},f_{d_{k,n}}) \, \leq \, \mbox{const.}\cdot \{l_{k,n}^{-2} + \Delta_{0,j(k),n}^2\}\,,
\end{align*}
\noindent where the constant does not depend on ${\bf X}^*$. Therein we have utilized condition (\ref{eq:bound}) as well as the Lipschitz continuity, positivity and boundedness of $\varphi$. We take the expectation of the sum of these terms over $k=0,\ldots,m-1$ which converges to zero uniformly in $\theta\in\Theta$ by the assumption on $m$ and the imposed convergence rates of the estimator $\hat{\theta}$. Then the asymptotic equivalence is evident by the argument (\ref{eq:tool}) from the proof of Lemma \ref{L:2} when replacing the data sets $\tilde{Z}$ and ${\bf W}$ by the data samples $(d_{k,n},D_{k,n})_{k=0,\ldots,m-1}$ and $(s''_{k,n},S''_{k,n})_{k=0,\ldots,m-1}$, respectively, and inserting the conditional densities of their components given ${\bf X}^*$. The sum is, of course, to be taken over $k=0,\ldots,m-1$ instead of $j\in J_n$.  \hfill $\square$  \\

Now we go over to experiments involving Poisson point processes (PPP).

\begin{defn}
In experiment ${\mathcal H}_n$ we observe ${\bf X}^*$ and independently two
independent Poisson point processes $X_l$ and $X_u$ whose domain is the Borel
$\sigma$-algebra of $\R^2$ and whose intensity functions equal
\begin{align*}
\lambda_l(x,y) & \, = \,  m \varphi(1) \sum_{k=0}^{m-1} \rho_{k,n}  {\bf 1}_{I_{k,n}}(x) {\bf 1}_{[-C_\Theta-1,\theta(\xi_{j(k),n})]}(y)\,, \\
\lambda_u(x,y) & \, = \,  m \varphi(-1) \sum_{k=0}^{m-1} \rho_{k,n}  {\bf 1}_{I_{k,n}}(x) {\bf 1}_{[\theta(\xi_{j(k),n}),C_\Theta+1]}(y)\,,
\end{align*}
and are hence locally constant. We recall that $C_\Theta$ is the uniform upper bound on $|\theta|$ in the parameter set $\Theta$.
\end{defn}
We define the extreme points of $X_l$ and $X_u$ in the strip $I_{k,n} \times \R$ by
\begin{align*}
X_{l,k} & \, = \, \inf \big\{y \in \R \, : \, X_l(I_{k,n}\times [y,\infty)) = 0\big\}\,, \\
X_{u,k} & \, = \, \sup \big\{y \in \R \, : \, X_u(I_{k,n}\times (-\infty,y]) = 0\big\}\,.
\end{align*}

\begin{lem} \label{L:12}
(a)\; The statistic $(X_{l,k},X_{u,k})$, $k=0,\ldots,m-1$, is sufficient for the whole empirical information contained in $X_l$ and $X_u$. \\
(b)\;  The distribution functions of $X_{l,k}$ and $X_{u,k}$ are equal to those
of $\max\{-C_\Theta-1,D_{k,n}\}$ and $\min\{C_\Theta+1,d_{k,n}\}$, respectively
where $d_{k,n}$ and $D_{k,n}$ are as in experiment ${\mathcal G}_n$. Moreover,
all $X_{l,k}$, $k=0,\ldots,m-1$, on the one hand and all $X_{u,k}$,
$k=0,\ldots,m-1$ on the other hand are independent.
\end{lem}

\noindent {\it Proof of Lemma \ref{L:12}:} (a)\; Let $X_0$ denote the PPP with the intensity function $\lambda_0 = {\bf 1}_{[0,1]\times [-C_\Theta-1,C_\Theta+1]}$. The probability measures generated by $X_0,X_l,X_u$ are denoted by ${\bf P}_0,{\bf P}_l,{\bf P}_u$, respectively. As the functions $\lambda_0,\lambda_l,\lambda_u$ are piecewise constant and the support of $\lambda_l$ and $\lambda_u$ is included in that of $\lambda_0$ the measure ${\bf P}_0$ dominates ${\bf P}_l$ and ${\bf P}_u$ and the corresponding Radon-Nikodym derivatives are equal to
\begin{align*}
\frac{d {\bf P}_l}{d {\bf P}_0}(X) & \, = \, \exp\Big\{ \int \log \frac{\lambda_l(x,y)}{\lambda_0(x,y)} dX(x,y) - \int \Big(\frac{\lambda_l(x,y)}{\lambda_0(x,y)} - 1\Big) \lambda_0(x,y) \, dx\, dy\Big\}\,, \\
\frac{d {\bf P}_u}{d {\bf P}_0}(X) & \, = \, \exp\Big\{ \int \log \frac{\lambda_u(x,y)}{\lambda_0(x,y)} dX(x,y) - \int \Big(\frac{\lambda_u(x,y)}{\lambda_0(x,y)} - 1\Big) \lambda_0(x,y) \, dx\, dy\Big\}\,,
\end{align*}
\noindent see e.g. Theorem 1.3 in \citeasnoun{K98} which apparently goes back to \citeasnoun{B71}. Therein $X$ may be viewed as an arbitrary counting process on the Borel $\sigma$-algebra of $[0,1]\times [-C_\Theta-1,C_\Theta+1]$. We write $\Gamma_\theta = \bigcup_{k=0}^{m-1} I_{k,n}\times (\theta(\xi_{j(k),n}),C_\Theta+1]$ and $\Phi = \bigcup_{k=0}^{m-1} I_{k,n}\times [-C_\Theta-1,\tilde{X}_{l,k}]$ where $\tilde{X}_{l,k}$ equals $X_{l,k}$ except that $X_l$ is changed into the general process $X$ in the definition. Then $d {\bf P}_l / d {\bf P}_0$ is equal to
\begin{align*}
\frac{d {\bf P}_l}{d {\bf P}_0}(X) & \, = \, {\bf 1}_{\{\emptyset\}}(\Gamma_\theta \cap \Phi) \cdot \exp\Big\{ \sum_{k=0}^{m-1} \log[\rho_{k,n} m \varphi(1)] X(I_{k,n}\times [-C_\Theta-1,C_\Theta+1])\Big\}\\ & \hspace{2cm}\cdot  \exp\Big\{-\sum_{k=0}^{m-1} (\theta(\xi_{j(k),n}) + C_\Theta + 1) \rho_{k,n} \varphi(1)\Big\} \exp(2C_\Theta+2)\,,
\end{align*}
\noindent where we have used that $X(I_{k,n}\times [-C_\Theta-1,C_\Theta+1]) = X(I_{k,n}\times [-C_\Theta-1,\theta(\xi_{j(k),n})])$ whenever $X(\Gamma_\theta) = 0$; and that $\Gamma_\theta$ and $\Phi$ are disjoint if and only if $X(\Gamma_\theta) = 0$. It follows from the Fisher-Neyman factorization theorem that the $X_{l,k}$, $k=0,\ldots,m-1$ represent a sufficient statistic for $X_l$. The corresponding assertion for the $X_{u,r}$ is proved analogously. \\

(b)\; We consider for $x\in [-C_\Theta-1,\theta(\xi_{j(k),n})]$ that
\begin{align*}
P[X_{l,k} \leq x] & \, = \, P[X_l(I_{k,n}\times (x,\infty)) = 0] \, = \, \exp\big(- (\theta(\xi_{j(k),n})-x)  \rho_{k,n} \varphi(1)\big) \\
& \, = \, P[D_{k,n} \leq x]\,.
\end{align*}
Clearly we have $P[X_{l,k} > \theta(\xi_{j(k),n})] = P[D_{k,n} > \theta(\xi_{j(k),n})] = 0$ and $P[X_{l,k} < -C_\Theta-1] = 0$ so that the distribution functions of $X_{l,k}$ and $\max\{-C_\Theta-1,D_{k,n}\}$ coincide. The claim that $X_{u,k}$ and $\min\{C_\Theta+1,d_{k,n}\}$ are identically distributed follows analogously. Finally the independence of the data $X_{l,k}$, $k=0,\ldots,m-1$ as well as of the data $X_{u,k}$, $k=0,\ldots,m-1$ follows from the fact that $X(A_0),\ldots,X(A_{m-1})$ are independent for all $A_k \subseteq I_{k,n} \times [-C_\Theta-1,C_\Theta+1]$ by the definition of the PPP. \hfill $\square$ \\

\begin{lem} \label{L:13}
For $m=O( n^{1-\delta})$, $\delta>0$, the total variation distance between the
distributions of $\big((\min\{C_\Theta+1,d_{k,n}\},\max\{-C_\Theta-1,D_{k,n}\})
: k=0,\ldots,m-1\big)$ and $\big((d_{k,n},D_{k,n}) : k=0,\ldots,m-1\big)$
converges to zero.
\end{lem}

\noindent {\it Proof of Lemma \ref{L:13}:} Due to the independence of the data the desired total variation distance is bounded from above by the sum of the total variation distances between the distributions of $d_{k,n}$ and $\min\{C_\Theta+1,d_{k,n}\}$ plus the corresponding distances between the distributions of $D_{k,n}$ and $\max\{-C_\Theta-1,D_{k,n}\}$ where $k=0,\ldots,m-1$. The total variation distance between $d_{k,n}$ and $\min\{C_\Theta+1,d_{k,n}\}$ is bounded by
%\begin{align*}
%& |E_\theta R(\min\{C_\Theta+1,d_{k,n}\}) - E_\theta R(d_{k,n})| \\ & \qquad \leq \, \big|E_\theta R(C_\Theta+1) {\bf %1}_{[C_\Theta+1,\infty)}(d_{k,n}) - E_\theta R(d_{k,n}) {\bf 1}_{[C_\Theta+1,\infty)}(d_{k,n})\big| \\
%& \qquad\leq \,
\[2 P[d_{k,n} \geq C_\Theta+1] \, \leq \, 2 \exp\big(-\mbox{const.}\cdot n/m\big)\,,\]
%\end{align*}
\noindent so that because of $m \leq \mbox{const.}\cdot n^{1-\delta}$ the sum of these terms for $k=0,\ldots,m-1$ tends to zero exponentially fast. The distributions of $\max\{-C_\Theta-1,D_{k,n}\}$ and $D_{k,n}$ are treated in the same way. \hfill $\square$. \\

\noindent Combining these two lemmata we obtain directly asymptotic equivalence.

\begin{cor}
Experiments ${\mathcal G}_n$ and ${\mathcal H}_n$ are asymptotically equivalent
for $m$ as in Lemma \ref{L:13}.
\end{cor}

\noindent We observe that the choice $m\asymp n^{2/3-\delta}$ for some
$\delta\in(0,1/6)$ meets all requirements imposed on $m$ so far and we
summarize our results.

\begin{prop} \label{P:1}
Select $m\asymp n^{2/3-\delta}$ for some $\delta\in(0,1/6)$ and suppose that there is an estimator $\hat\theta$, based on the data ${\bf X}^*$ alone, which satisfies (\ref{eq:bound}) and
\[
\sup_{\theta \in \Theta}\,  \sup_{x\in [0,1]} \, E_\theta|\hat{\theta}(x) - \theta(x)|^2 \, = \, O(n^{-1-\delta}).
\]
Then we have asymptotic equivalence between experiments ${\mathcal C}_n$ and
${\mathcal H}_n$. Moreover, if we have additionally
\[\sup_{\theta \in \Theta}\,  \sup_{x\in [0,1]} \, E_\theta|\hat{\theta}'(x) - \theta'(x)| \, = \, o(n^{-1/3-\delta}),
\]
then also ${{\mathcal A}_n}^*$ and ${\mathcal H}_n$ are asymptotically
equivalent.
\end{prop}

\section{Localization of the PPP model} \label{5}
%%%%%%%%%%%%%%%%%%%%%%%%%%%%%%%%%%%%%%%%%%%%%%%%%

The processes $X_l$ and $X_u$ in the experiment ${\mathcal H}_n$ have step
functions as their intensity boundaries which approximate continuous functions
as $m$ tends to infinity. Therefore we consider now the experiment where ${\bf
X}^*$ and independently two PPP with boundary function $\theta$ are observed.

\begin{defn}
In experiment ${\mathcal I}_n$ we observe ${\bf X}^*$ and independently two
independent PPP $X_{1,0}$ and $X_{2,0}$ with intensities
\begin{align} \nonumber
\lambda_{1,0}(x,y) & \, = \, (n/2) f_\eps(1) f_D(x) {\bf 1}_{[-C_\Theta-1,\theta(x)]}(y)\,, \\ \label{eq:K}
\lambda_{2,0}(x,y) & \, = \, (n/2) f_\eps(-1) f_D(x) {\bf 1}_{[\theta(x),C_\Theta+1]}(y)\,.
\end{align}
\end{defn}

\begin{prop}\label{PropHI}
We impose the conditions of Lemma \ref{L:1} and, in addition, that for all $\theta\in\Theta$, we have
\begin{equation} \label{eq:bound2} \sup_{x\in [0,1]} |\hat{\theta}'(x)| \, \leq \, 2 \sup_{\theta\in\Theta} \sup_{x\in [0,1]} |\hat{\theta}'(x)|\,, \qquad \mbox{a.s.}
\end{equation}
Then the experiments ${\mathcal H}_n$ and ${\mathcal I}_n$ are asymptotically
equivalent.
\end{prop}

\noindent {\it Proof of Proposition \ref{PropHI}:} First, we show asymptotic
equivalence of the experiment ${\mathcal H}_n$ with the experiment ${{\mathcal
H}_n}'$ in which one observes the data $({\bf X}^*,\tilde{X}_1,\tilde{X}_2)$
where $\tilde{X}_1$ and $\tilde{X}_2$ are PPP with the intensity functions
\begin{align*}
\tilde{\lambda}_1(x,y) & \, = \, (n/2) f_\eps(1) f_D(x) {\bf 1}_{[-C_\Theta-1,\theta(x)-\hat{\theta}'(x) (\xi(x)-x)]}(y)\,, \\
\tilde{\lambda}_2(x,y) & \, = \, (n/2) f_\eps(-1) f_D(x) {\bf 1}_{[\theta(x)-\hat{\theta}'(x) (\xi(x)-x),C_\Theta+1]}(y)\,,
\end{align*}
\noindent conditionally on ${\bf X}^*$, respectively. Here, $\hat{\theta}'$ denotes the pilot estimator from Lemma \ref{L:1} based on the data set ${\bf X}^*$; and we write $\xi(x)$ for the centre of that interval $I_{k,n}$ which contains the element $x$.
%\begin{lem} \label{L:5.1}
%Under the conditions of Lemma \ref{L:1} the experiments I and J are asymptotically equivalent.
%\end{lem}

%\noindent {\it Proof of Lemma \ref{L:5.1}:}

By a similar argument as in (\ref{eq:A1.0}), it suffices to show that the
expected Hellinger distance between the distribution of $\tilde{X}_1$ and $X_l$
on the one hand and $\tilde{X}_2$ and $X_u$ on the other hand converges to
zero. We shall now employ a general formula bounding the Hellinger distance
between two PPP laws $P_1,P_2$ with respective intensities
$\lambda_1,\lambda_2$ by the (generalized) Hellinger distance of the
intensities ; when $P$ denotes the law of the PPP with intensity
$\lambda=\lambda_1+\lambda_2$, we derive from the likelihood expression
\begin{align}
H^2& (P_1,P_2) = 2\Big(1-E_\theta \exp\Big(\int \frac12(\log(\lambda_1/\lambda)+\log(\lambda_2/\lambda))dX - \int \Big(\frac{\lambda_1+\lambda_2}{2\lambda}-1\Big)\lambda\Big)\Big) \nonumber\\
&= 2\Big(1-\Big\{E_\theta\exp\Big(\int \log\sqrt{\lambda_1\lambda_2}/\lambda\,dX -\int (\sqrt{\lambda_1\lambda_2}/\lambda-1)\lambda\Big)\Big\} \nonumber\\
& \hspace{9cm} \cdot \exp\Big(-\int (\sqrt{\lambda_1}-\sqrt{\lambda_2})^2/2 \Big)\Big)\nonumber \\
&=2\left(1-\exp\left(-\int (\sqrt{\lambda_1}-\sqrt{\lambda_2})^2/2  \right)\right)\label{eq:HellPPP}\\
&\le \int (\sqrt{\lambda_1}-\sqrt{\lambda_2})^2\,,\nonumber
\end{align}
where we have used the fact that the Radon-Nikodym-derivative of the PPP-law
with intensity $\sqrt{\lambda_1\lambda_2}$ with respect to $P$ integrates to
one under $P$, see also \citeasnoun{LCY00} for a related result. Thus we bound
the Hellinger distance between the intensities of $\tilde X_1$ and $X_l$ by

\begin{align*}
\int (\sqrt{\lambda_l}-\sqrt{\tilde\lambda_1})^2  \, \leq\,  \mbox{const.}\cdot n \,\sum_{k=0}^{m-1} \Big\{\int_{I_{k,n}}& \big|\theta(\xi(x)) - \theta(x) + \hat{\theta}'(x) (\xi(x) - x)\big| dx \\ & \, + \, \int_{I_{k,n}} \Big|f_D(x) - m \int_{I_{k,m}} f_D(y) dy\Big|^2 dx\Big\}\,,
\end{align*}
\noindent where the constant does not depend on ${\bf X}^*$. As $f_D$ is assumed to be Lipschitz on $[0,1]$ the  latter term contributes to the asymptotic order by the deterministic upper bound $O\big(n m^{-2}\big)$ independently of $\theta$. Then we apply the expectation to the above expression and we obtain
$$ O\big(n m^{-2}\big) +  \mbox{const.}\cdot n m^{-1} \sup_{\theta\in\Theta} \sup_{x\in [0,1]} E_\theta \big|\hat{\theta}'(x) - \theta'(x)\big| =  o(1)\,, $$
\noindent as a uniform upper bound. Together with the same bound for the
Hellinger distance, conditionally on ${\bf X}^*$, between the intensities of
$\tilde X_2$ and $X_u$ this implies asymptotic equivalence between ${\mathcal
H}_n$ and ${{\mathcal H}_n}'$ again by arguments as in (\ref{eq:tool}).

For any two-dimensional Borel set $B$ let us define the pointwise shifted version
$$ \hat{B} \, = \, \big\{(x,y) \in \R^2\, : \, \big(x,y+\hat{\theta}'(x) [\xi(x) - x]\big) \in B\big\}\,, $$
\noindent and the processes $\overline{X}_j(B) \, = \, \tilde{X}_j(\hat{B})$, $j=1,2$, conditionally on the data set ${\bf X}^*$. Note that $\hat{B}$ is a Borel set as well whenever the shift function $\hat{\theta}'(\cdot) [\xi(\cdot) - \cdot]$ is piecewise continuous on the intervals $I_{k,n}$. Then $\overline{X}_j$ represents a PPP with the shifted intensity function
\begin{align*}
\overline{\lambda}_1(x,y) & \, = \, (n/2) \varphi(1) f_D(x) {\bf 1}_{[-C_\Theta-1+\hat{\theta}'(x) (\xi(x)-x),\theta(x)]}(y) \\
\overline{\lambda}_2(x,y) & \, = \, (n/2) \varphi(-1) f_D(x) {\bf 1}_{[\theta(x),C_\Theta+1+\hat{\theta}'(x) (\xi(x)-x)]}(y)\,.
\end{align*}
\noindent Note that this transformation is invertible as long as the data set
${\bf X}^*$ is available. Therefore, the experiment ${{\mathcal H}_n}''$ of
observing ${\bf X}^*$ and $\overline{X}_j$, $j=1,2$ independently is equivalent
to the experiment ${{\mathcal H}_n}'$.

By the imposed upper bound on the estimator $\hat{\theta}'$ we may assume that
$$ \sup_{\theta\in\Theta} \sup_{x\in[0,1]} |\hat{\theta}'(x)| |\xi(x)-x| \, \leq \, 1/2\,, $$
\noindent for $m$ sufficiently large. Hence, the observation of $\overline{X}_j$, $j=1,2$, is equivalent with the observation of two conditionally independent Poisson processes $\overline{X}_{j,1}$ and $\overline{X}_{j,2}$ with the intensity functions
\begin{align*}
\overline{\lambda}_{1,1}(x,y) & \, = \, (n/2) \varphi(1) f_D(x) {\bf 1}_{[-C_\Theta-1/2,\theta(x)]}(y)\,, \\
\overline{\lambda}_{1,2}(x,y) & \, = \, (n/2) \varphi(1) f_D(x) {\bf 1}_{[-C_\Theta-1+\hat{\theta}'(x) (\xi(x)-x),-C_\Theta-1/2)}(y)\,, \\
\overline{\lambda}_{2,1}(x,y) & \, = \, (n/2) \varphi(-1) f_D(x) {\bf 1}_{[\theta(x),C_\Theta+1/2]}(y)\,, \\
\overline{\lambda}_{2,2}(x,y) & \, = \, (n/2) \varphi(-1) f_D(x) {\bf 1}_{(C_\Theta+1/2, C_\Theta+1+\hat{\theta}'(x) (\xi(x)-x)]}(y)\,,
\end{align*}
\noindent Thus all processes $\overline{X}_{j,i}$, $i,j=1,2$, are independent.
Also we realize that the processes $\overline{X}_{1,2}$ and
$\overline{X}_{2,2}$ represent conditionally ancillary statistics given the
data set ${\bf X}^*$ as $\overline{\lambda}_{1,2}$ and
$\overline{\lambda}_{2,2}$ do not explicitly depend on $\theta$, but are fixed
by knowledge of ${\bf X}^*$ for $n$ sufficiently large. Therefore, the
observation of ${\bf X}^*$ and $\overline{X}_{j,1}$, $j=1,2$ is sufficient for
complete empirical information contained in experiment ${{\mathcal H}_n}''$. On
the other hand we may also add two independent PPP $\overline{X}_{j,3}$,
$j=1,2$ with the intensity functions
\begin{align*}
\overline{\lambda}_{1,3}(x,y) & \, = \, (n/2) \varphi(1) f_D(x) {\bf 1}_{[-C_\Theta-1,-C_\Theta-1/2)}(y)\,, \\
\overline{\lambda}_{2,3}(x,y) & \, = \, (n/2) \varphi(-1) f_D(x) {\bf 1}_{(C_\Theta+1/2,C_\Theta+1]}(y)\,,
\end{align*}
\noindent which are totally uninformative. Combining the independent
processes $\overline{X}_{j,1}$ and $\overline{X}_{j,3}$ whose intensity
functions are supported on (almost) disjoint domains for both $j=1,2$,
the considered experiment is equivalent to the experiment ${{\mathcal I}_n}$.  \hfill $\square$ \\

\section{Final proof} \label{6}
%%%%%%%%%%%%%%%%%%%%%%%%%%%%%%%

In this section, we combine all results derived in the previous sections in
order to complete the proof of Theorem \ref{T:1}. For simplicity we suppose
that $n$ is even. By Proposition \ref{Prop:0} with sample size
$n/2$, there exists an estimator $\hat\theta$ based on the data ${\bf X} = {\bf
X}^*$ from experiment ${{\mathcal A}_n}$ which satisfies the conditions of
Proposition \ref{P:1}, e.g. by choosing $\delta=\alpha/2$. Therefore,
experiments ${\mathcal A}_n$ and ${\mathcal I}_n$ are asymptotically equivalent
by Propositions \ref{P:1} and \ref{PropHI}. The conditions (\ref{eq:bound}) and
(\ref{eq:bound2}) are satisfied when truncating the range of $\hat{\theta}$ and
$\hat{\theta}'$ suitably without losing validity of Proposition \ref{Prop:0}.
Therein, note that the uniform upper bounds on $\theta\in \Theta$ as well as on
its derivative are known. Then we set ${{\mathcal  A}_n}^* = {\mathcal I}_n$ by using the
processes $X_{1,0}$ and $X_{2,0}$ as the data set ${\bf X}^*$ and let ${\bf X}$
take the role of the data ${\bf Y}'$ from experiment ${{\mathcal A}_n}$. Note
that all of our arguments from the previous sections remain valid when
transforming the responses with even instead of odd observation number.
Applying Propositions \ref{P:1} and \ref{PropHI} again, we obtain asymptotic
equivalence of the experiments ${{\mathcal  I}_n}$ and ${{\mathcal J}_n}$ where
the latter model just consists of $X_{1,0}$ and $X_{2,0}$ and two independent
copies $X_{1,0}^*$ and $X_{2,0}^*$. The likelihood process of experiment
${\mathcal J}_n$ and experiment ${\mathcal B}_n$ turns out to be the same,
using Theorem 1.3 in \citeasnoun{K98} as in the proof of Lemma \ref{L:12}, such
that ${\mathcal J}_n$ and ${\mathcal B}_n$ are equivalent experiments. The
concrete equivalence mapping is given by looking at the sum of the processes
$X_j=X_{j,0}+X_{j,0}^\ast$, $j=1,2$, in one direction and by splitting the
point masses in $X_j$ randomly and independently with probability one half into
point masses for $X_{j,0}$ and $X_{j,0}^\ast$ ({\it thinning} of a PPP) for the
other equivalence direction.

\begin{comment}In experiment ${\mathcal I}_n$ we now set $\bf X^\ast:=(X_{1,0},X_{2,0})$ and by Proposition \ref{Prop:0}, applied to experiment ${\mathcal B}_n$ with sample size $n/2$, there exists an estimator $\hat\theta^\ast$ based on $\bf X^\ast$ which satisfies the conditions of Proposition \ref{P:1}, when the roles of $\bf X$ and $\bf X^\ast$ are interchanged. As stressed before, these estimation rates are all we need from $\bf X$ respectively $\bf X^\ast$ for the proof such that experiment ${\mathcal I}_n$ is asymptotically equivalent to experiment ${\mathcal J}_n$ where we observe $\bf X^\ast$ and independently two independent PPP $X_{1,0}^\ast,X_{2,0}^\ast$ with the same intensities $\lambda_{1,0}$ and $\lambda_{2,0}$ as in $\bf X$.
\end{comment}

\section{Discussion} \label{7}
%%%%%%%%%%%%%%%%%%%%%%%%%%%%

\subsection{General remarks} \label{7.0}

We have shown asymptotic equivalence of nonparametric regression with
non-regular additive errors and the observation of two specific independent
PPP. Our result also yields that those nonparametric regression models are
asymptotically equivalent to each other as long as the corresponding error
densities have the same jump sizes at $-1$ and $+1$ and are Lipschitz
continuous and positive within the interval $(-1,1)$ -- regardless of the
specific shape of the density inside its support. This unifies the asymptotic
theory for these experiments and properties such as asymptotic minimax bounds,
adaptation, superefficiency can be studied simultaneously for those models. At
least after suitable linear correction by a pilot estimator, local minima and
maxima are asymptotically sufficient for inference in these models.

The limiting Poisson point process model ${\mathcal B}_n$ exhibits a
fascinating new geometric structure. According to \eqref{eq:HellPPP}, the
squared Hellinger distances between observations with parameters
$\theta_1,\theta_2\in\Theta$ is given by
\[ H^2(P_{\theta_1},P_{\theta_2})
=2\Big(1-\exp\Big(-\frac n2(f_\eps(-1)+f_\eps(+1))\int\abs{\theta_1(x)-\theta_2(x)}f_D(x)\,dx\Big)\Big).
\]
Setting $\norm{g}_{L^1_X}:=\int\abs{g(x)}f_D(x)\,dx$, the squared Hellinger
distance is thus equivalent to an $L^1$-distance
\begin{equation}\label{eq:H2L1}
H^2(P_{\theta_1},P_{\theta_2}) \asymp
n \{f_\eps(-1)+f_\eps(+1)\}\norm{\theta_1-\theta_2}_{L_X^1}.
\end{equation}
In contrast, for nonparametric regression with regular errors the continuous
limit model is a Gaussian shift where the corresponding squared Hellinger
distance is equivalent to $n\sigma^{-2}\norm{\theta_1-\theta_2}_{L_X^2}^2$ with
$\sigma^2=\text{Var}(\eps_{j,n})$. While it is well known that the standard
parametric rate improves from $n^{-1/2}$ to $n^{-1}$, the nonparametric view
reveals that we face here an $L^1_X$-topology instead of the usual Hilbert
space $L^2_X$-structure. As discussed below, this different Banach space
geometry is even visible at the level of minimax rates, which are in general
worse than for regular nonparametric regression with sample size $n^2$. A
boundary behaviour of the error density $f_\eps$ other than finite jumps will
imply a different Hellinger topology, in particular the whole range of
$L^p_X$-geometries, $p\in (0,\infty)$, might arise, whose statistical
consequences will be far-reaching and remain to be explored in detail.

\begin{comment}
There are still important questions which we leave open for future
research. For instance, we believe that for error densities, whose support is
only bounded from either the right or the left side, are asymptotically
equivalent to the observation of only one of the processes $X_1$ or $X_2$,
respectively. Also we conjecture that asymptotic equivalence can be extended to
the case of error densities with non-sharp support boundaries by an appropriate
modification of the intensity function of the PPP. Furthermore, we suppose that
our results are extendable to regression models with random covariates with
appropriate design densities.

Anwendungen

einseitiger Tr\"ager

Randverhalten wie $x^\alpha$, $\alpha\in (-1,1)$

Abh\"angigkeit der Konstruktion von $f_\eps$ (bzw. auch $f_D$)

Random design

....?
\end{comment}

\subsection{A nonparametric lower bound}\label{SecLB}

Let us apply the asymptotic equivalence result to study nonparametric lower
bounds for all models in ${\mathcal A}_n$ and for ${\mathcal B}_n$,
simultaneously. We content ourselves here with rate results, but we track
explicitly the dependence on the total jump size $J:=f_\eps(-1)+f_\eps(1)$ and
the design density $f_D$.

\begin{prop}\label{PropLB}
In the PPP model ${\mathcal B}_n$, but with $\theta$ from the parameter space
\[ \Theta_{s,L}:=\{\theta\in C^s([0,1])\,|\,
\norm{\theta}_s\le L\},\quad s,L>0
\]
with generalized H\"older norm
\[
\norm{g}_{s}:=\max_{k=0,1,\ldots,\floor{s}}\norm{g^{(k)}}_\infty+
\sup_{x\not=y}\frac{\abs{g(x)-g(y)}}{\abs{x-y}^{s-\floor{s}}}
\]
the following lower bound for the pointwise loss in estimating $\theta$ and its
derivatives at $x_0\in[0,1]$ holds uniformly in $J:=f_\eps(-1)+f_\eps(1)$, $x_0$ and $f_D(x_0)$
\[
\liminf_{n\to\infty}\inf_{\hat\theta_n}\sup_{\theta\in\Theta_{s,L}}P_\theta\Big(\abs{\hat\theta_n^{(k)}(x_0)-\theta^{(k)}(x_0)}
\ge c_0 \frac{L^{(k+1)/(s+1)}}{(nJf_D(x_0))^{(s-k)/(s+1)}}\Big)\ge \frac{2-\sqrt{3}}{4}>0
\]
with $c_0>0$, where the infimum is taken over all estimators in ${\mathcal
B}_n$ and $k=0,1,\ldots,\floor{s}$.
\end{prop}

By asymptotic equivalence and the boundedness of the involved loss function $1\{\abs{\hat\theta_n^{(k)}(x_0)-\theta^{(k)}(x_0)}
\ge c L^{(k+1)/(s+1)}(nJf_D(x_0))^{-(s-k)/(s+1)}\}$, this result immediately generalizes to the regression
experiments ${\mathcal A}_n$ provided the regularity $s$ is larger than two.
Moreover, by Markov's inequality it also applies to $p$-th moment risk. We thus
have:

\begin{cor}
For estimators $\hat\theta_n$ in experiment ${\mathcal A}_n$ with $\theta\in
\Theta_{s,L}\subset\Theta$ and $s>2$, $L>0$ we have for all $p>0$,
$k=0,1,\ldots,\floor{s}$ the lower bound
\[
\liminf_{n\to\infty} L^{-(k+1)/(s+1)}(nJf_D(x_0))^{(s-k)/(s+1)} \inf_{\hat\theta_n}\sup_{\theta\in\Theta_{s,L}}\big(E_\theta\abs{\hat\theta_n^{(k)}(x_0)-\theta^{(k)}(x_0)}^p\big)^{1/p}
\ge c_1 
\]
for some constant $c_1>0$. %uniformly in $J$, $x_0\in[0,1]$ and $f_D(x_0)$.
\end{cor}

\begin{proof}[Proof of the Proposition \ref{PropLB}]
Let us fix $k\in\{0,1,\ldots\floor{s}\}$. By Theorem 2.2(ii) in \citeasnoun{T09} it suffices to find
$\theta_1,\theta_2\in\Theta_{s,L}$ with
\[\abs{\theta_1^{(k)}(x_0)-\theta_2^{(k)}(x_0)}\ge
L^{(k+1)/(s+1)}(nJf_D(x_0))^{-(s-k)/(s+1)}
\]
and Hellinger distance of the
corresponding observation laws satisfying $H(P_{\theta_1},P_{\theta_2})\le 1$.

We choose some kernel function $K\in\Theta_{s,1}$ with $\int_{-1}^1K(x)\,dx=1$,
$K^{(k)}(0)>0$ and support in $[-1/2,1/2]$ and we set $\theta_1(x)=0$,
$\theta_2(x)=Lh^sK((x-x_0)/h)$ with $h=(LnJf_D(x_0))^{-1/(s+1)}$ (using one-sided kernel versions near the boundary). Then for $n$
sufficiently large we have $\theta_1,\theta_2\in \Theta_{s,L}$ and moreover by \eqref{eq:H2L1}
\[ H^2(P_{\theta_1},P_{\theta_2})=(1+o(1))nJ \int_{-1}^1
\abs{\theta_1(x)-\theta_2(x)}f_D(x)\,dx
\]
and the integral satisfies $\int_{-1}^1 \abs{\theta_2(x)}f_D(x)\,dx=(L+o(1))
h^{s+1}f_D(x_0)$ as $h\to 0$. We conclude that $H(P_{\theta_1},P_{\theta_2})$
converges to one for $n\to\infty$. The result therefore follows from
\[ \abs{\theta_2^{(k)}(x_0)-\theta_2^{(k)}(x_0)}=K^{(k)}(0)L^{(k+1)/(s+1)}(nJf_D(x_0))^{-(s-k)/(s+1)}.\]
\end{proof}

The rate $L^{(k+1)/(s+1)}n^{-(s-k)/(s+1)}$ instead of
$L^{(k+1/2)/(s+1/2)}\sqrt{n}^{-(s-k)/(s+1/2)}$ for regular nonparametric
regression is obviously due to the $L^1_X$-bound on $\theta_2$ instead of the
squared $L^2_X$-bound. Let us mention that a careful study of our upper bound proof in Proposition \ref{Prop:0} will also yield the same dependence on $L=C_\Theta$ for regularity $s=2+\alpha$ and $k\in\{0,1\}$.  More geometrically, we can establish a lower bound for
estimating a linear functional $L(\theta)$ by maximising $L(\theta)$ over
$\theta\in\Theta_{s,L}$ with $\norm{\theta}_{L^1_X}\le 1/(nJ)$. In the scale of
Besov spaces $B^\alpha_{p,p}$ with norms $\norm{\cdot}_{\alpha,p}$,
$\alpha\in\R$, $1\le p\le\infty$, we have $\norm{\theta}_{L^1}\ge
\norm{\theta}_{-1,\infty}$ by duality from $\norm{\theta}_{L^\infty}\le
\norm{\theta}_{1,1}$. Here, we can therefore expect to maximise
$L(\theta)=\theta^{(k)}(x_0)$ as far as the interpolation inequality
\[ \norm{\theta}_{k,\infty}\le
\norm{\theta}_{-1,\infty}^{(s-k)/(s+1)}\norm{\theta}_{s,\infty}^{(k+1)/(s+1)}
\le\text{const.}(nJ)^{-(s-k)/(s+1)}L^{(k+1)/(s+1)}
\]
permits. This is in fact achieved by the choice of $\theta_2$ above, involving
also the localized value $f_D(x_0)$. In the corresponding regular nonparametric
regression model the Hellinger constraint is given by
$\norm{\theta}_{L^2_X}^2\le \sigma^2/n$ and we use
$\norm{\theta}_{L^2}\ge\norm{\theta}_{-1/2,\infty}$ by duality from
$\norm{\theta}_{L^2}\le \norm{\theta}_{1/2,1}$ to obtain the interpolation
inequality
\[ \norm{\theta}_{k,\infty}\le
\norm{\theta}_{-1/2,\infty}^{(s-k)/(s+1/2)}\norm{\theta}_{s,\infty}^{(k+1/2)/(s+1/2)}
\le\text{const.}(\sigma^{-2}n)^{-(s-k)/(2s+1)}L^{(k+1/2)/(s+1/2)},
\]
which similarly reveals the minimax rate in the regular case. Very roughly, we might therefore say that the PPP noise induces a regularity $-1$ in the H\"older scale, while the Gaussian white noise leads to the higher regularity $-1/2$.
In analogy with $\sigma/\sqrt{n}$ in the regular case we might call $1/(nJ)$
the {\it noise level} for the regression problem with irregular
noise and $nJf_D(x_0)$ the {\it effective local sample size} at $x_0$.

\subsection{One-sided frontier estimation}

In many of the applications mentioned in the introduction, the noise density $f_\eps$ has just one jump and not two as in our model ${\mathcal A}_n$. We want to stress that our proof of asymptotic equivalence can also cover the one-jump case. To make the analogy clear, let us assume that $f_\eps$ is still a density on $[-1,1]$ with $f_\eps(-1)>0$ and $f_\eps(1)=0$. Instead of positivity and Lipschitz continuity, we now require $f_\eps$ to be Lipschitz continuous and Hellinger differentiable on $[-1,1]$, i.e. $\sqrt{f_\eps}$ is weakly differentiable with derivative in $L^2([-1,1])$. Note that $f_\eps$ can then be extended to a function $\phi$ on the real line with the same local properties. All other properties of the model ${\mathcal A}_n$ are kept the same.

For the pilot estimator in this model we can obtain the same convergence rates when we select that admissible local polynomial which is the smallest at $x_0$. Lemma \ref{L:1} remains the same, while in Definition \ref{defD} of experiment ${\mathcal D}_n$ we adjust only the left boundary of the density and set
\begin{align*}
f_{W,j}(x) & \, = \, \varphi(x) \Big(\int_{\Delta_{0,j,n}-1}^{\Delta_{0,j,n}+1} \varphi(t) dt\Big)^{-1} 1_{[\Delta_{0,j,n}-1,\infty)}(x)\,,\qquad j\in J_n\,.
\end{align*}
Lemma \ref{L:2} then remains true as well, using the Hellinger differentiability in the proof instead of the uniform positivity. From the form of the density of $W$ we conclude this time that the local minima $s_{k,n}=\min\{W_{j,n}:x_{j,n}\in I_{k,n}\}$, $k=0,\ldots,m-1$, are conditionally sufficient. Then the remaining results remain all valid if we just consider $s_{k,n}$ instead of $(s_{k,n},S_{k,n})$ and merely the upper PPP model. Consequently, this establishes asymptotic equivalence with the PPP $X_2$ of experiment ${\mathcal B}_n$. In this PPP model the regression function $\theta$ appears as the lower frontier of a Poisson point process with intensity $f_D(x)nf_\eps(-1)$ on its epigraph. Frontier estimation where the support of $f_\eps$ is on $[-1,\infty)$ or $(-\infty,1]$, respectively, can be treated analogously. In a general model the case of a regular density $f_\eps$ with finitely many jumps at known locations might be treated, which should also be asymptotically equivalent to suitable PPP models.

\subsection{Counterexample for regularity one}

We give a short argument that for equidistant design $x_{j,n}=\frac{j-1}{n-1}$
and parameter classes $\Theta$ where the target function $\theta\in\Theta$ is
required to satisfy $\norm{\theta'}\le C$ for some $C>0$ the experiments
${\mathcal  A}_n$ and ${\mathcal B}_n$ are not asymptotically equivalent.
Whether H\"older classes of order $1+\alpha$ instead of $2+\alpha$ suffice as
parameter sets for establishing asymptotic equivalence remains a challenging
open question.

Let us consider the function $f_n(x)=C(\pi(n-1))^{-1}\sin(\pi (n-1) x)$ so that
$\norm{f_n'}_\infty=C$ holds for all $n\ge 1$. Now observe that $f_n$ satisfies
$f_n(x_{j,n})=0$ for all $j=1,\ldots,n$. This means in particular that in the
regression experiment ${\mathcal A}_n$ the observations with regression
function $f_n$ cannot be distinguished from those with zero regression
function. In experiment ${\mathcal B}_n$, however, a test between $H_0:
\theta=0$ and $H_1:\theta=f_n$ of the form $T_n=1\{X_1([0,1]\times\R^+)>0\text{
or } X_2([0,1]\times\R^-)>0\}$ satisfies $P_0(T_n=0)=1$ and
\begin{align*}
P_{f_n}(T_n=1)=1-\exp\Big(-n\int_0^1\abs{f_n(x)}dx\Big) & =1-\exp(-2C\pi^{-2}n(n-1)^{-1})\\ & \to 1-\exp(-2C/\pi^2)>0\,,
\end{align*}
for $n\to\infty$. Consequently, testing between $H_0$ and $H_1$ in experiment
${\mathcal B}_n$ is possible with non-trivial power uniformly over $n$. This
implies that experiments ${\mathcal A}_n$ and ${\mathcal  B}_n$ are
asymptotically non-equivalent.

\thebibliography{99}

\harvarditem{Brown and Low}{1996}{BL96}
Brown, L.D. and Low, M. (1996). Asymptotic equivalence of nonparametric regression and white noise. {\it Ann. Statist.} {\bf 24}, 2384--2398.

\harvarditem{Brown et al.}{2002}{BCLZ02}
Brown, L., Cai, T., Low, M. and Zhang, C.-H. (2002). Asymptotic equivalence theory for nonparametric regression with random design. {\it Ann. Statist.} {\bf 30}, 688--707.

\harvarditem{Brown et al.}{2010}{BCZ10}
Brown, L., Cai, T., Zhou, H.H. (2010). Nonparametric regression in exponential families. {\it Ann. Statist.} {\bf 38}, 2005--2046.

\harvarditem{Brown}{1971}{B71}
Brown, M. (1971). Discrimination of Poisson processes. {\it Ann. Math. Statist.} {\bf 42}, 773--776.

\harvarditem{Carter}{2007}{C07}
Carter, A. (2007). Asymptotic approximation of nonparametric regression experiments with unknown variances. {\it Ann. Statist.} {\bf 35}, 1644--1673.

\harvarditem{Carter}{2009}{C10}
Carter, A. (2009). Asymptotically sufficient statistics in nonparametric regression experiments with correlated noise. {\it J .Prob. Statist.} {\bf 2009}, ID 275308 (19 pages).

\harvarditem{Chernozhukov and Hong}{2004}{CH04}
Chernozhukov, V. and Hong, H. (2004). Likelihood estimation and inference in a class of nonregular econometric models. {\it Econometrica} {\bf 72}, 1445--1480.

\harvarditem{DeVore and Lorentz}{1993}{DL93}
DeVore, R.A. and Lorentz, G.G. (1993). {\it Constructive Approximation}, Grundlehren Series {\bf 303}, Springer, Berlin.

\harvarditem{Gijbels et al.}{1999}{GMPS99} Gijbels, I., Mammen, E., Park, B. and Simar, L. (1999). On estimation of monotone and concave frontier functions, {\it J. Amer. Statist. Assoc.} {\bf 94}, 220-228.

\harvarditem{Grama and Nussbaum}{1998}{GN98}
Grama, I. and Nussbaum, M. (1998). Asymptotic equivalence for nonparametric generalized linear models. {\it Prob. Th. Rel. Fields} {\bf  111}, 167--214.

\harvarditem{Grama and Nussbaum}{2002}{GN02}
Grama, I. and Nussbaum, M. (2002). Asymptotic equivalence for nonparametric regression. {\it Math. Meth. Stat.} {\bf  11}(1), 1--36.

\harvarditem{Hall and van Keilegom}{2009}{HK09}
Hall, P. and van Keilegom, I. (2009). Nonparametric ``regression'' when errors are positioned at end-points. {\it Bernoulli} {\bf 15}, 614--633.

\harvarditem{Janssen and Marohn}{1994}{JM94}
Janssen, A. and Marohn, D.M. (1994). On statistical information of extreme order statistics, local extreme value alternatives and Poisson point processes. {\it J. Multivar. Anal.} {\bf 48}, 1--30.

\harvarditem{Karr}{1991}{K91}
Karr, A.F. (1991). {\it Point Processes and Their Statistical Inference}, 2nd ed., Marcel Dekker, New York.

\harvarditem{Knight}{2001}{K01}
Knight, K. (2001). Limiting Distributions of Linear Programming Estimators. {\it Extremes} {\bf 4}, 87--103.

\harvarditem{Korostelev and Tsybakov}{1993}{KT93}
Korostelev, A.P. and Tsybakov, A.B. (1993). {\it Minimax Theory of Image Reconstruction}, Lecture Notes in Statistics {\bf 82}, Springer, New York.

\harvarditem{Kutoyants}{1998}{K98}
Kutoyants, Y.A. (1998). {\it Statistical Inference for Spatial Poisson Processes}, Lecture Notes in Statistics {\bf 134}, Springer, New York.

\harvarditem{Le Cam}{1964}{LC64}
Le Cam, L.M. (1964). Sufficiency and approximate sufficiency. {\it Ann. Math. Statist.} {\bf 35}, 1419--1455.

\harvarditem{Le Cam and Yang}{2000}{LCY00}
Le Cam, L.M. and Yang, G.L. (2000), {\it Asymptotics in Statistics, Some Basic Concepts}, 2nd ed., Springer.

\harvarditem{M\"uller and Wefelmeyer}{2010}{MW10}
M\"uller, U.U. and Wefelmeyer, W. (2010). Estimation in nonparametric regression with nonregular errors. {\it Comm. Statist. Theo. Meth.} {\bf 39}, 1619--1629.

\harvarditem{Nussbaum}{1996}{N96}
Nussbaum, M. (1996). Asymptotic equivalence of density estimation and Gaussian white noise. {\it Ann. Statist.} {\bf 24}, 2399--2430.

\harvarditem{Rei{\ss}}{2008}{R08}
Rei{\ss}, M. (2008). Asymptotic equivalence for nonparametric regression with multivariate and random design. {\it Ann. Statist.} {\bf 36}, 1957--1982.

\harvarditem{Tsybakov}{2009}{T09} Tsybakov, A. B. (2009). {\it Introduction to Nonparametric Estimation}, Springer Series in Statistics.

\harvarditem{van de Geer}{2006}{vdG06} van de Geer, S.A. (2006). {\it Empirical Processes in M-Estimation}, Reprint, Cambridge University Press, New York.

\end{document}